\documentclass{amsart}
\usepackage{latexsym,amssymb,amsmath,enumerate,graphicx,mathrsfs,booktabs,lscape}
\usepackage{float}
\usepackage{epstopdf}
\usepackage{bbm}
\usepackage{rotating}
\usepackage{rotate}
\usepackage{pst-node}
\usepackage{pstricks}
\usepackage{srcltx}
\usepackage{psfrag}
\usepackage[all,cmtip]{xy}
\DeclareMathAlphabet{\mathcalligra}{T1}{calligra}{m}{n}
\begin{document}

\theoremstyle{plain}
  \newtheorem{theorem}{Theorem}[section]
  \newtheorem{proposition}[theorem]{Proposition}
  \newtheorem{lemma}[theorem]{Lemma}
  \newtheorem{corollary}[theorem]{Corollary}
  \newtheorem{conjecture}[theorem]{Conjecture}
\theoremstyle{definition}
  \newtheorem{definition}[theorem]{Definition}
  \newtheorem{example}[theorem]{Example}
  \newtheorem*{example*}{Example}
  \newtheorem{observation}[theorem]{Observation}
  \newtheorem{convention}[theorem]{Convention}
  \newtheorem{observations}[theorem]{Observations}
  \newtheorem{question}[theorem]{Question}
  \newtheorem{mainquestion}{Question}
  \renewcommand{\themainquestion}{\Alph{mainquestion}}
  \newtheorem{problem}[theorem]{Problem}
  \newtheorem*{problem*}{Problem}
 \theoremstyle{remark}
  \newtheorem{remark}[theorem]{Remark}
\def\ff{{\mathbf f}}
\def\ZZ{{\mathbb Z}}
\def\QQ{{\mathbb Q}}
\def\CC{{\mathbb C}}
\def\NN{{\mathbb N}}
\def\comp{\mathrm{c}}
\def\Kbar{{\overline{K}}}
\def\type{\mathrm{type}}
\def\WH{\mathrm{WH}}
\def\isom{\cong}
\def\homotopic{\simeq}
\def\Lcm{{\mathrm{lcm}}}

\makeatletter
\newcommand{\superimpose}[2]{%
{\ooalign{$#1\@firstoftwo#2$\cr\hfil$#1\@secondoftwo#2$\hfil\cr}}}
\makeatother
\newcommand{\bigcupdot}{\mathpalette\superimpose{{\bigcup}{\cdot}}}
\newcommand{\cupdot}{\mathpalette\superimpose{{\cup}{\cdot}}}
\newcommand{\epsilonA}{\delta}
\newcommand{\deltaA}{\epsilon}
\def\Y{{\mathbf Y}}
\def\SY{{S\mathbf Y}}
\def\YF{{\mathbf Y F}}
\def\il{{\text{Smith-eigenvalue}}}
\def\pd{{\frac{\partial}{\partial p_1}}}
\def\K{{\mathbb C}}
\def\E{{\mathcal E}}
\def\A{{\mathscr A}}
\def\P{{\mathscr{P}}}
\def\z{{\mathrm{erase}}}
\def\zero{{\mathrm{zero}}}
\def\Q{{\mathbf Q}}
\def\GL{\mathrm{GL}}
\def\symm{\mathfrak{S}}
\def\codim{{\mathrm{codim}}}
\def\bmu{{\boldsymbol\mu}}
\def\Ell{\ell}
\def\lex{\mathrm{lex}}
\def\word{\mathrm{word}}
\def\Col{\mathrm{Col}}
\def\tilde{\widetilde}
\def\CP{\Pi}
\def\diff{\setminus}
\def\Par{P(W,R)}
\def\coker{\mathrm{coker}}
\def\im{\mathrm{im}}
\def\diag{\mathrm{diag}}
\def\rank{\mathrm{rank}}
\def\Res{\mathrm{Res}}
\def\Ind{\mathrm{Ind}}
\def\triv{{\mathbbm{1}}}
\def\Irr{\mathrm{Irr}}
\def\ones{{\mathrm{ones}}}
\def\Gal{{\mathrm{Gal}}}
\def\Fix{\mathrm{Fix}}
\def\Pw{\mathcal P_W}
\def\type{\mathrm{type}}
\def\Class{{\mathbf{Cl}}}
\def\CoShe{\mathrm{CoShe}}
\def\sd{\mathrm{sd}}
\def\pDelta{\Delta^{U}}
\def\lk{\mathrm{lk}}
\def\L{\mathscr L}
\def\calF{\mathcal{F}}
\def\calP{\mathcal{U}}
\def\calP{\mathcal{U}}
\def\calB{\mathcal{B}}
\def\calT{T}
\def\calJ{\mathcal{J}}
\def\affine{\mathrm{{affine}}}
\def\pPi{\Pi^U}
\def\Stab{\mathrm{Stab}}
\def\St{\mathrm{St}}
\def\Supp{\mathrm{Supp}}
\def\Span{\mathrm{Span}}
\def\Def{\stackrel{\textbf{def}}{=}}
\def\GDef{\stackrel{\phantom{\text{def}}}{=}}
\def\Dash{\text{\textbf{---}}}
\def\Conv{\mathrm{Conv}_{\mathbb R}}
\def\AffSpan{\mathrm{AffSpan}}
\def\LinSpan{\mathrm{Span}}
\def\Hull{\mathrm{Hull}}
\def\Face{\mathrm{Face}}
\def\Sd{\mathrm{Sd}}
\def\Hilb{\mathrm{Hilb}}
\def\Hom{\mathrm{Hom}}
\def\Sh{\mathrm{Sh}}
\renewcommand{\wp}{p}
\newcommand{\Wedge}{{\textstyle{\bigwedge}}}
\newcommand{\tableskip}{}
\newcommand{\circled}[1]{{\raisebox{.5pt}{\textcircled{\raisebox{-.9pt}{#1}}}}}
\title[Eigenspace arrangements]{Eigenspace arrangements of reflection groups}
\author{Alexander R. Miller}
\email{mill1966@math.umn.edu}
\address{ School of Mathematics\\
University of Minnesota\\
Minneapolis, MN 55455} 
\begin{abstract}  
The lattice of intersections of reflecting hyperplanes of a complex reflection group $W$ 
may be considered as the poset of $1$-eigenspaces of the elements of $W$.  
In this paper we replace $1$ with an arbitrary eigenvalue and study the topology and 
homology representation of the resulting poset.  
After posing the main question of whether this poset is shellable, 
we show that all its upper intervals are geometric lattices, and then answer the 
question in the affirmative for the infinite family $G(m,p,n)$ 
of complex reflection groups, and the first 31 of the 34 exceptional groups, by constructing 
CL-shellings.  In addition, we completely determine when these eigenspaces of $W$ form a $K(\pi,1)$ (resp. free) 
arrangement.

For the symmetric group, we also extend the combinatorial model 
available for its intersection lattice to all other eigenvalues by introducing 
\emph{balanced partition posets}, 
presented as particular upper order ideals of Dowling lattices, study 
the representation afforded by the top (co)homology group, and 
give a simple map to the posets of pointed $d$-divisible partitions.
\end{abstract}
\thanks{Partially supported by NSF grant DMS-1001933}
\maketitle

\section{Main question and results}\label{Section:Introduction}
Let $V$ denote an $n$-dimensional $\mathbb C$-vector space.  A \emph{reflection} 
in $V$ is any non-identity element $r$ in $\GL(V)$ of finite order that fixes some hyperplane 
$H_r$, and a finite subgroup $W$ of $\GL(V)$ is called a \emph{reflection group} 
if it is generated by reflections.  

\begin{example*}\label{Example:Sn}
The action of the symmetric group $\mathfrak S_n$ on $[n]:=\{1,2,\ldots, n\}$ gives 
rise to a faithful action on $\mathbb C^n$ via $\sigma(e_i)=e_{\sigma(i)}$, where 
$e_1,e_2,\ldots, e_n$ denote the standard basis vectors.  Since the transpositions 
act as reflections and generate the group, this representation realizes 
$\mathfrak S_n$ as a reflection group in $\GL(\mathbb C^n)$.  We shall 
refer to it as the \emph{defining representation}.
\end{example*}

For an element $g\in W$ and root of unity $\zeta$, let $V(g,\zeta)$ denote the 
$\zeta$-eigenspace of $g$ in $V$.  
Define $E(W,\zeta)$ to be the $W$-poset (partially ordered set) of all 
such $\zeta$-eigenspaces $\{V(g,\zeta)\}_{g\in W}$ ordered by reverse inclusion, with 
$W$-action given by $h\cdot V(g,\zeta)=V(hgh^{-1},\zeta)$.  
Choosing $\zeta=1$ recovers the lattice $\L_W$ of intersections of reflecting hyperplanes 
for $W$; see~\cite[Lemma 4.4]{Orlik}.  The minimal elements of $E(W,\zeta)$ 
(i.e., inclusion-maximal $\zeta$-eigenspaces of $W$) 
are the focus of Springer's theory of regular elements~\cite{Springer}, and each has dimension equal to 
the number $a(d)$ of \emph{degrees} $d_1,d_2,\ldots, d_n$ 
of $W$ that are divisible by $d$, the order of $\zeta$; see Proposition~\ref{Prop:Springer:Transitive} below.  
When $W$ is crystallographic, the poset $E(W,\zeta)$ itself 
appears in Brou\'e, Malle, and Michel's $\Phi$-Sylow theory~\cite{BMM}.

\begin{example*}
The degrees of $\mathfrak S_4$ are $1,2,3,4$, and those of the dihedral group $I_2(4)$ (whose cardinality is $8$) are $2,4$.  Hence the following table.

\begin{center}\begin{tabular}{lccccc}
\toprule
$d$ & $1$ & $2$ & $3$ & $4$ & $\geq 5$ \\
\midrule
$a(d)$ for $\mathfrak S_4$ & 4 & 2 & 1 & 1 & 0\\

$a(d)$ for $I_2(4)$ & 2 & 2 & 0 & 1 & 0\\
\bottomrule
\end{tabular}
\end{center}
When $a(d)=0$, the poset $E(W,\zeta)$ has only one element, the $0$-dimensional subspace.  For all other listed cases, we have 
provided the Hasse diagram of $E(W,\zeta)$ in Figures~\ref{Figure:A1}-\ref{Figure:B}, labeling each eigenspace 
by linear equations that define it, and adorning maximal eigenspaces with an additional 
integer label.  For example, the maximal eigenspace $E_4$ labeled \circled{$4\hskip .015cm$} in Figure~\ref{Figure:A3} is 
the $\zeta$-eigenspace for the 3-cycle permutation $g=(1,4,3)$.
Note that the $d=1,2$ cases for $I_2(4)$ coincide, since the scalar matrix $-1$ is an element of $I_2(4)$; see Corollary~\ref{Cor:Springer} below.
\end{example*}

\begin{figure}[hbt]
\centering
\includegraphics[scale=1.1]{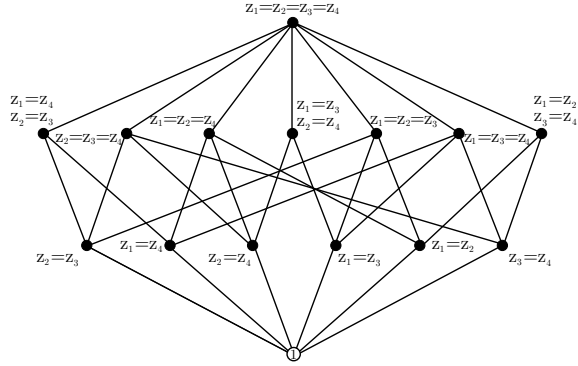}
\caption{The poset $E(\mathfrak S_4,1)=\L_{\mathfrak S_4}$ of $1$-eigenspaces for $\mathfrak S_4\subset\GL(\mathbb C^4)$.}\label{Figure:A1}
\end{figure}

\begin{figure}[hbt]
\centering
\includegraphics[width=\textwidth]{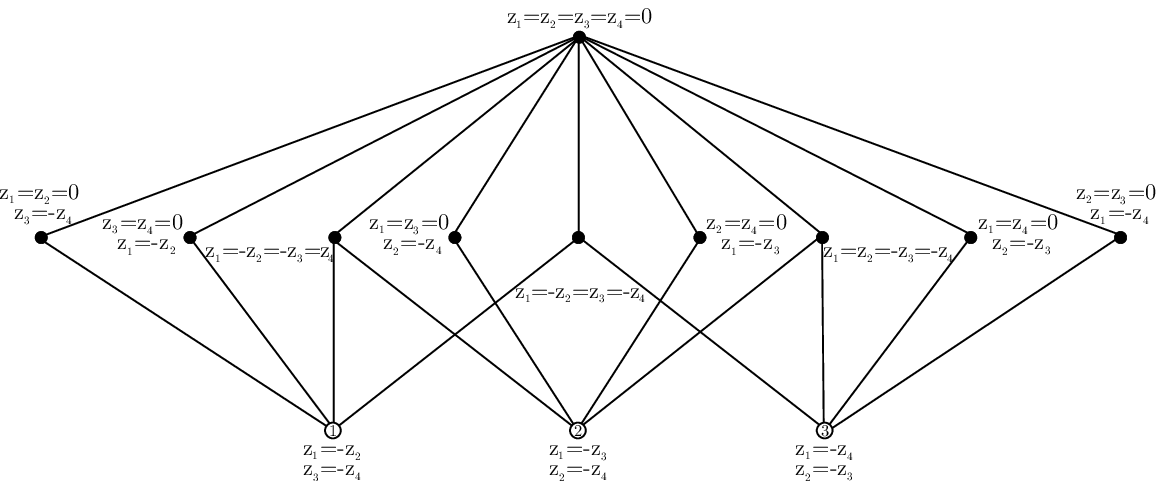}
\caption{The poset $E(\mathfrak S_4,-1)$ of $(-1)$-eigenspaces for $\mathfrak S_4\subset\GL(\mathbb C^4)$.}\label{Figure:A2}
\end{figure}

\begin{figure}[H]
\centering
\includegraphics[scale=1.1]{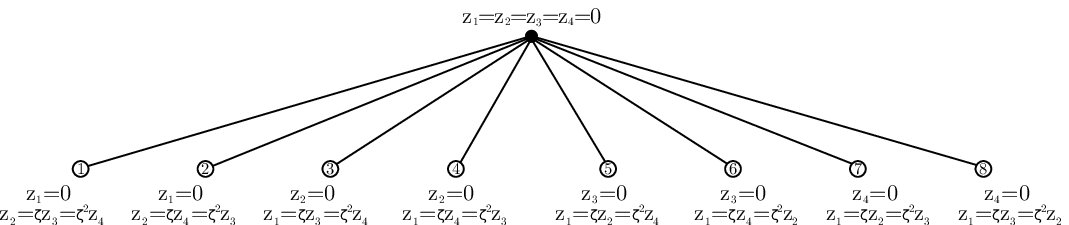}
\caption{The poset $E(\mathfrak S_4,\zeta_3)$ of $\zeta$-eigenspaces for $\zeta$ any primitive 3rd root of unity and $\mathfrak S_4\subset\GL(\mathbb C^4)$.}\label{Figure:A3}
\end{figure}

\begin{figure}[H]
\centering
\includegraphics[scale=1.1]{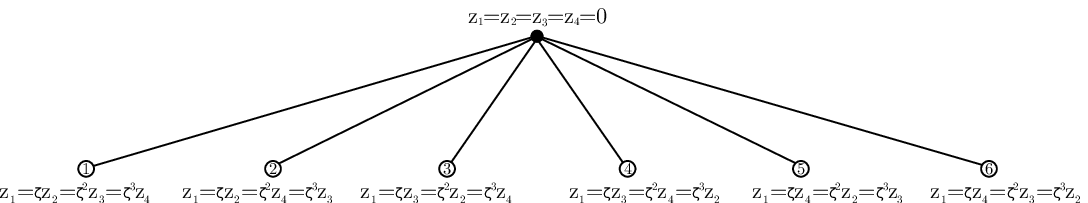}
\caption{The poset $E(\mathfrak S_4,\zeta)$ of $\zeta$-eigenspaces for $\zeta$ any primitive 4rd root of unity and 
$\mathfrak S_4\subset\GL(\mathbb C^4)$.}\label{Figure:A4}
\end{figure}

\begin{figure}[H]
\centering
\includegraphics[scale=1.1]{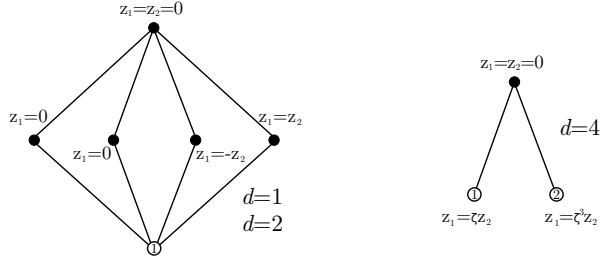}
\caption{The poset $E(W,\zeta)$ of $\zeta$-eigenspaces for the dihedral group $I_2(4)$ of order $8$, 
and $\zeta$ of order $d=1,2,4$.}\label{Figure:B}
\end{figure}

This paper concerns the following problem of Lehrer and Taylor~\cite[Problem 7]{Lehrer}.
\begin{problem*}[Lehrer-Taylor]  
Study connections between the structure and representations of $W$ and the 
topology of the posets $E(W,\zeta)$.
\end{problem*}

In the case of $E(\mathfrak S_n,1)$, Stanley~\cite{Stanley:Aspects} used the work of Hanlon~\cite{Hanlon:Fix} to obtain 
an explicit expression for the top homology representation, which combined with Klyachko's work  
to establish a connection with the \emph{Lie representation} $\mathrm{Lie}_n$; see~\cite{Klyachko,Joyal}.  
Lehrer and Solomon~\cite{Lehrer:Solomon} 
extended Stanley's result 
and conjectured an analogue for all other finite Coxeter groups.  Hanlon's work~\cite{Hanlon} on \emph{Dowling lattices} 
gives an alternate extension and provides the top homology character of $E(W,1)$ for $W=G(m,1,n)$, the 
complex reflection group of $n\times n$ monomial matrices whose nonzero entries are $m^{\text{th}}$ roots of unity.  
N.~Bergeron~\cite{Bergeron} gave a type-$B$ analogue of the abovementioned Lie correspondence that was subsequently generalized to 
Dowling lattices by Gottlieb and Wachs~\cite{GottliebWachs}.  In addition to the above, analogous results have been 
obtained for various subposets of $E(\mathfrak S_n,1)$; see~\cite{Wachs:Tools}.  However, the author is unaware of any analogous 
results for $\zeta\neq 1$, even for $\mathfrak S_n$.    

\begin{mainquestion}\label{Question:Main:Shelling}
Is the following true for every $\zeta$ and every reflection group $W$?\\
\begin{tabular}{ll}
{\bf (Weak version)} & $E(W,\zeta)$ is homotopy Cohen-Macaulay.\\
{\bf (Strong version)} & $\widehat{E(W,\zeta)}$ is CL-shellable.
\end{tabular}
\end{mainquestion}

\noindent
The fact that CL-shellability implies homotopy CM-ness is well-known; see \S\ref{Section:Preliminaries}. \medskip

Our first main result answers affirmatively the strong version of Question~\ref{Question:Main:Shelling}
for all irreducible complex reflection groups except types $E_6,E_7,E_8$ 
(which, in the Shephard-Todd classification, are $G_{35},G_{36},G_{37}$).
Since the question reduces (see \S\ref{Section:General} below) to the case where $W$ acts irreducibly, 
only these three Weyl groups remain.

\begin{theorem}\label{Thm:Main:Shelling}
Let $W=G(m,p,n)$ or one of the first 31 exceptional groups $G_4,G_5,\ldots, G_{34}$, 
and let $\zeta$ be a primitive $d^{\text{th}}$ root of unity.
Then $\widehat{E(W,\zeta)}$ is CL-shellable.  In particular, the order complex $\Delta(\overline{E(W,\zeta)})$ is 
a pure bouquet of spheres.
\end{theorem}

Our second main result is central to the first.

\begin{theorem}\label{Thm:Main:Geometric}
Let $W$ be a reflection group and let $\zeta$ be a primitive $d^{\text{th}}$ root of unity.  Then
\begin{equation}
E(W,\zeta)=\{E\cap X\ :\ X\in \L_W\ \text{and}\ E\in E(W,\zeta)\ \text{maximal}\}.
\label{Main:Geometric:Equation}
\end{equation}
In particular, each upper interval $[x,\hat{1}]$ in $E(W,\zeta)$ is a geometric lattice.
\end{theorem}

\noindent
We will see in Section~\ref{Section:General} below that this theorem also has interesting consequences 
of its own, namely, that $E(W,\zeta)$ depends only on $d$ (see Theorem~\ref{Thm:Strong:Independence} for a much sharper result) 
and is built from copies (conjugates) of the intersection 
lattice $\L_{W(d)}$ of Lehrer and Springer's reflection subquotient $W(d)$ (see Proposition~\ref{Thm:Lehrer:Springer}).

 \medskip

We also answer the natural $K(\pi,1)$ and freeness questions for eigenspaces,  
that is, we determine exactly when the complement of the proper $\zeta$-eigenspaces of a
reflection group $W$ is $K(\pi,1)$, and 
exactly when the $\zeta$-eigenspaces of codimension 1 form a free (hyperplane) arrangement.   
For brevity, define $\mathcal A(W,d)$ to be the set of all 
\emph{proper} $\zeta$-eigenspaces of the reflection group $W$ \emph{that are maximal under inclusion}, 
for a fixed but arbitrary choice of primitive $d^{\text{th}}$ root of unity $\zeta$, so that 
the complement of the proper $\zeta$-eigenspaces of $W$ is
\begin{equation}\mathcal M(W,d)=\mathbb C^n\diff\textstyle{\bigcup_{X\in\mathcal A(W,d)}}X.\end{equation}
The two questions were resolved affirmatively for the highly nontrivial case of the \emph{reflection arrangement} $\mathcal A(W,1)$ by 
Bessis and Terao, respectively, and in contrast to Question~\ref{Question:Main:Shelling}, the non-classical 
case $\mathcal A(W,d)\neq\mathcal A(W,1)$ is surprisingly simple.  
\begin{theorem}\label{Thm:Main:KPi1}
Let $W\subset \GL(\mathbb C^n)$ be a complex reflection group and let $d>0$.\\Then the following are equivalent:
\begin{enumerate}[(i)]
\item  One has $a(d)=n$ or $a(d)=n-1$.  That is, $\mathcal A(W,d)$ contains a hyperplane.
\item  The arrangement $\mathcal A(W,d)$ is a free hyperplane arrangement.
\item  The complement $\mathcal M(W,d)$ is a $K(\pi,1)$ space.
\end{enumerate}
Furthermore, $\mathcal A(W,d)=\mathcal A(W,1)$ if and only if $a(d)=n$.
\end{theorem}
The reader should be warned that no \emph{new} examples of free or $K(\pi,1)$ arrangements 
occur in Theorem~\ref{Thm:Main:KPi1}.  
Nevertheless, the result 
is essentially all one could hope for, and provides a pleasant extension of the 
classical picture to arbitrary eigenvalues.

\tableofcontents
{\vskip -.3cm}
\noindent{\bf Acknowledgements.}  This work was partly supported by NSF grant DMS-1001933, and 
forms part of the author's doctoral work at the University of Minnesota under the supervision of 
Victor Reiner, whom the author thanks for many helpful conversations.  He is also 
grateful to Gustav Lehrer and Donald Taylor for writing~\cite{Lehrer}, in which they 
posed the above problem, and to Anders Bj\"orner and Volkmar Welker for their helpful comments.

\section{Preliminaries}\label{Section:Preliminaries}   
Recall that a \emph{$G$-poset} 
is a poset $P$ with a $G$-action that preserves order, i.e., $x<y$ implies $gx<gy$ for all $g\in G$ and $x,y\in P$. 
It is \emph{bounded} if it has a unique minimal element (called the 
\emph{bottom element} and denoted $\hat{0}$) 
and a unique maximal element (called the \emph{top element} and denoted $\hat{1}$).
Write $P\cup\{\hat{0}\}$ for the poset obtained from $P$ by adjoining 
a new element $\hat{0}$, regardless of whether $P$ has a bottom element.  Similarly, $P\cup\{\hat{1}\}$ 
is obtained by adjoining a new element $\hat{1}$.   Appending both yields $\widehat{P}:=P\cup\{\hat{0},\hat{1}\}$. 

The \emph{order complex} of a poset $P$ is the (abstract) simplicial complex $\Delta(P)$ consisting of all 
totally ordered sets $x_1<x_2<\cdots<x_i$ in $P$.  Because it has a cone point (and is therefore contractible) if $P$ 
has a top or bottom element, one often considers the \emph{proper part} $\overline{P}:=P\diff\{\hat{0},\hat{1}\}$ of the poset, 
which is simply $P$ if neither a top nor bottom element is present.

Recall that a (finite) simplicial complex $\Delta$ is $k$-connected if its homotopy groups $\pi_j(\Delta)$ 
are trivial for $0\leq j\leq k$.  Define the \emph{link} of a face $F\in \Delta$ to be the 
subcomplex 
\[\lk_{\Delta}(F)=\{G\in\Delta\ :\ G\cup F\in \Delta,\ G\cap F=\varnothing\},\]
and say $\Delta$ is \emph{homotopy Cohen-Macaulay} (abbreviated HCM) if 
for each face $F\in\Delta$ the link 
$\lk_{\Delta}(F)$ is 
$(\dim\lk_{\Delta}(F)-1)$-connected.  For $\mathbf k$ a field or $\mathbb Z$, 
the simplicial complex is said to be \emph{Cohen-Macaulay} over 
$\mathbf{k}$ (abbreviated CM/$\mathbf k$ or simply CM when $\mathbf k=\mathbb Z$) if for each 
face $F\in \Delta$ the homology groups 
$\widetilde{H}_i(\lk_{\Delta}(F); \mathbf{k})$ vanish for $i<\dim\lk_{\Delta}(F)$.  
The property of being Cohen-Macaulay is topologically invariant and implies that the 
complex is homologically a bouquet of ($\dim\Delta$)-spheres, whereas the stronger property of 
being homotopy Cohen-Macaulay is not topologically invariant but implies that the complex 
is homotopically a bouquet of $(\dim\Delta)$-spheres; see~\cite{Munkres} and~\cite[p. 117]{Quillen}, 
respectively, or surveys~\cite{Bjorner:Methods, Wachs:Tools}.

Though there are many techniques for establishing Cohen-Macaulayness, we shall 
be concerned with \emph{(pure) CL-shellability.}  
A simplicial complex $\Delta$ which is pure $d$-dimensional (i.e., each maximal face under 
inclusion has dimension $d$) is said to be \emph{shellable} if its maximal faces (called \emph{facets}) 
can be ordered $F_1,F_2,\ldots, F_\ell$ so that for each $k$ the subcomplex generated by the first 
$k$ facets intersects the $(k+1)$st facet in a pure $(d-1)$-dimensional subcomplex.  
A poset $P$ is called \emph{shellable} (resp. \emph{HCM, CM}) if its order complex 
$\Delta(P)$ is shellable (resp. HCM, CM).  
Finally, a \emph{CL-shellable}\footnote{This recursive formulation of CL-shellability 
is due to Bj\"orner and Wachs~\cite{BW}.} poset is a bounded poset that 
\emph{admits a recursive atom ordering}, as defined in Section~\ref{Section:Geometric}.
The following implications for poset shellability are strict:
\[
\text{CL-shellable}
\Rightarrow
\text{shellable}
\Rightarrow 
\text{HCM} 
\Rightarrow
\text{CM}
\Rightarrow
\text{CM/{$\mathbf k$}}
\Rightarrow
\text{CM/{$\mathbb Q$}},
\]
Note that a poset $P$ is shellable 
(resp. HCM, CM) if and only if $\overline{P}$, or just $P\diff\{\hat{1}\}$, is shellable (resp. HCM, CM).

\section{General reductions}\label{Section:General} 
Shephard and Todd classified all irreducible reflection groups in~\cite{Shephard:Todd}. 
There are 34 exceptional groups in their classification, labeled $G_4,G_5,\ldots, G_{37}$, and 3 infinite families, 
explained below:
\begin{itemize}
\item $W(A_n)\subset \GL(\mathbb C^n)$.
\item $G(m,p,n)$ with $m>1$, $p$ a divisor of $m$, and $(m,p,n)\neq (2,2,2)$.
\item $C_m:=G(m,1,1)$.
\end{itemize}
$W(A_n)$ denotes the representation of $\mathfrak S_{n+1}$ one obtains from the defining 
representation of \S\ref{Section:Introduction} after modding out 
by the fixed space $\mathbb C(e_1+e_2+\cdots+e_{n+1})$.  For $\bmu_m$ the 
set of $m^{\text{th}}$ roots of unity (and $p$ a divisor or $m$), 
the group $G(m,p,n)$ consists of all $n\times n$ monomial 
matrices with nonzero entries in 
$\bmu_m$ whose product lies in $\bmu_{m/p}$.
General reflection groups decompose into irreducible ones as follows.

\begin{proposition}[Theorem 1.27 in \cite{Lehrer}]\label{Prop:Decomposition}  
Let $W\subset \GL(V)$ be a reflection group.  
Let $V_1,\ldots, V_k$ denote the nontrivial irreducible submodules of $V$, 
so that the restriction $W_i$ of $W$ to $V_i$ is irreducible.  Then
$W\cong W_1\times \cdots \times W_k$ and 
one has an orthogonal sum decomposition 
$V=V^W\oplus V_1\oplus \cdots\oplus V_k$ with respect to a chosen $W$-invariant positive definite Hermitian form.
\end{proposition}

\begin{corollary}\label{Cor:Eigenspace:Decomposition}
Maintain the notation of Proposition~\ref{Prop:Decomposition} and let $\zeta$ be a root of unity.  Then 
\begin{equation}E(W,\zeta)\cong E(W_1,\zeta)\times\cdots\times E(W_k,\zeta).\label{Eq:Factorization}\end{equation}
\end{corollary}

In particular, the family of posets $E(W,\zeta)$ obtained by letting $W$ vary over the three infinite families above is the same 
as that obtained by letting $W$ vary over the single infinite family $G(m,p,n)$.  Another consequence is 
that Question~\ref{Question:Main:Shelling} and Theorem~\ref{Thm:Main:Geometric} reduce to irreducible reflection groups.

\begin{corollary}\label{Cor:Reduce}  
Maintain the notation of Corollary~\ref{Cor:Eigenspace:Decomposition}.  Then the following hold.
\begin{enumerate}[(i)]
\item  $E(W,\zeta)$ is HCM (resp. CM) if and only if each $E(W_i,\zeta)$ is HCM (resp. CM).\label{Eq:Decomp:1}
\item  $\widehat{E(W,\zeta)}$ is CL-shellable if and only if each $\widehat{E(W_i,\zeta)}$ is CL-shellable.\label{Eq:Decomp:2}
\end{enumerate}
\end{corollary}

\begin{proof}
Note that each $E(W_i,\zeta)$ has a top element $\hat{1}_i$.  Then~\eqref{Eq:Decomp:1} 
follows from~\eqref{Eq:Factorization} and a homeomorphism of Quillen (see~\cite[Ex. 8.1]{Quillen} 
and~\cite[Thm. 5.1(b)]{Walker}), while~\eqref{Eq:Decomp:2} follows 
from~\eqref{Eq:Factorization} and~\cite[Thm. 10.16]{Shellable:Nonpure:II}.
\end{proof}

Because of the important role that \emph{maximal} eigenspaces play in what follows, we make the following convention before proceeding.

\begin{convention}
A \emph{maximal} $\zeta$-eigenspace for $W$ is one that is not properly contained in any other.  
Because such a space is minimal with respect to the poset order of $E(W,\zeta)$ given by 
reverse inclusion, in order to avoid confusion we shall always take \emph{minimal} and \emph{maximal} to 
be with respect to inclusion when dealing with subspaces.  
For example, ``$E\in E(W,\zeta)$ maximal'' thus means that $E$ is not properly 
contained in any $V(g,\zeta)$.
\end{convention}

A reflection group $W\subset\GL(V)$ acts on the algebra of polynomial functions $S=S(V^*)$ 
via $gf(v)=f(g^{-1}v)$, and Shephard and Todd showed that the subalgebra $S^W$ of $W$-fixed 
polynomials is again polynomial, generated by $n:=\dim V$ algebraically 
independent homogeneous polynomials $f_1,f_2,\ldots, f_n$, called \emph{basic invariants}, the degrees of which 
are uniquely determined by the group $W$ and denoted $d_1\leq d_2\leq \cdots\leq d_n$.  We shall always assume an indexing 
such that $\deg(f_i)=d_i$.  For $d>0$, 
write $a(d)$ for the number of $d_i$ divisible by $d$ as in \S\ref{Section:Introduction}; see Springer~\cite{Springer}.

\begin{proposition}[Springer]\label{Prop:Springer:Intersection}
Let $W$ be a reflection group, let $\zeta$ be a primitive $d^{\text{th}}$ root of unity, and let $f_1,f_2,\ldots, f_n$ 
be a set of basic invariants for $W$.  Set $H_i:=\{v\in V\ :\  f_i(v)=0\}$.  Then one has $\ \bigcup_{g\in W}V(g,\zeta)=\bigcap_{d\, \nmid\, d_i}H_i$.
\end{proposition}

Consequentially, the collection of maximal $\zeta$-eigenspaces of 
$W$ depends only on (the group and) the multiset of degrees 
$d_i$ that are divisible by $d$, which we shall denote by  
$A(d):=\{d_i\ :\ d\mid d_i\}$, so that $a(d)=|A(d)|$; see Theorem~\ref{Thm:Strong:Independence} below for a
much sharper result.

\begin{corollary}\label{Cor:Independent:Maximals}
Let $W$ be a reflection group.  Let $\zeta$ and $\zeta'$ be roots of unity of orders $d$ and $d'$
such that $A(d)=A(d')$.  Then the set of maximal $\zeta$-eigenspaces 
of $W$ coincides with the set of maximal $\zeta'$-eigenspaces of $W$.
\end{corollary}

\begin{proposition}[Springer]\label{Prop:Springer:Transitive}
Let $W$ be a reflection group, let $\zeta$ be a primitive $d^{\text{th}}$ root of unity, and let $E,E'\in E(W,\zeta)$ 
be maximal eigenspaces.  Then
\begin{enumerate}[(i)]
\item  $\dim E=a(d)$, and\label{Springer:1}
\item  $gE=E'$ for some $g\in W$.\label{Springer:2}
\end{enumerate}
\end{proposition}

\begin{corollary}\label{Cor:Springer}
Let $W$ be a reflection group and let $\zeta$ be a primitive $d^{\text{th}}$ root of unity.  Then 
the following are equivalent.
\begin{enumerate}[(i)]
\item  $E(W,\zeta)=\L_W$.\label{Cor:Springer:2}
\item  $\mathbb C^n\in E(W,\zeta)$.\label{Cor:Springer:3}
\item  $a(d)=n$.\label{Cor:Springer:4}
\end{enumerate}
\end{corollary}

\begin{proof}
Clearly \eqref{Cor:Springer:2}$\Rightarrow$\eqref{Cor:Springer:3}$\Leftrightarrow$\eqref{Cor:Springer:4}.  
Assume $\mathbb C^n\in E(W,\zeta)$ so that $\zeta\in W$, and hence $\zeta^{-1} W=W$.  Writing 
$V(g,\zeta)=V(\zeta^{-1}g,1)$ for each $g\in W$, it follows that $E(W,\zeta)=\L_W$.
\end{proof}
\pagebreak
\begin{theorem}[Lehrer-Springer~\cite{Lehrer:Springer,Lehrer:Springer:Canada}]\label{Thm:Lehrer:Springer}  
Let $W$ be a reflection group, let $\zeta$ be a primitive $d^{\text{th}}$ root of unity, and let 
$E\in E(W,\zeta)$ be maximal with normalizer
\begin{align*}
N_W(E)&=\{g\in W\ :\ gE\subseteq E\}
\intertext{and centralizer} 
Z_W(E)&=\{g\in W\ :\ gv=v\text{ for all }v\in E\}.
\end{align*}  
Then $\overline{N}:=N_W(E)/Z_W(E)$ acts as a reflection group on $E$, and the following hold:
\begin{enumerate}[(i)]
\item If $f_1,f_2,\ldots, f_n$ form a set of basic invariants for $W$, then the restrictions $f_i|_E$ of those whose degree $d_i$ is divisible by $d$ 
form a set of basic invariants for $\overline{N}$.
\item The reflecting hyperplanes of $\overline{N}$ on $E$ are the intersections of $E$ with 
the reflecting hyperplanes of $W$ that do not contain $E$.\label{LS:Hyperplanes}
\item If $W$ is irreducible, then $\overline{N}$ acts irreducibly on $E$.
\item $W(d):=\overline{N}$ is uniquely determined by $W$ and $d$, up to conjugation by $W$.
\end{enumerate}
\end{theorem}

\begin{proposition}\label{Prop:Maximal}
Let $W$ be a reflection group and let $\zeta$ be a root of unity.  
Then one has an inclusion 
\begin{equation}
E(W,\zeta)\subseteq \{E\cap X\ :\ X\in\L_W,\ E\in E(W,\zeta)\ \text{maximal}\}
\label{Inclusion:Equation}
\end{equation}
and the following are equivalent.
\begin{enumerate}[(i)]
\item  Equality in~\eqref{Inclusion:Equation}.\label{Prop:General:1}
\item  There exists $E\in E(W,\zeta)$ maximal such that for all $X\in \L_W$ one has\\ $E\cap X\in E(W,\zeta)$.\label{Prop:General:2}
\item  For every $E\in E(W,\zeta)$ maximal and $X\in \L_W$ one has $E\cap X\in E(W,\zeta)$.\label{Prop:General:3}
\item  There exists $E\in E(W,\zeta)$ maximal such that $[E,\hat{1}]= \L_{N_W(E)/Z_W(E)}$.  \label{Prop:General:4}
\item  For every $E\in E(W,\zeta)$ maximal, one has $[E,\hat{1}]= \L_{N_W(E)/Z_W(E)}$\label{Prop:General:5}
\end{enumerate}
\end{proposition}

\begin{proof}
For the inclusion~\eqref{Inclusion:Equation}, let $Y\in E(W,\zeta)$ and choose a maximal $\zeta$-eigenspace $E$ such that $Y\subseteq E$.  
Write $E=V(g,\zeta)$ and $Y=V(h,\zeta)$ for some $g,h\in W$.  Then $v\in Y$ if and only if 
$hv=\zeta v=gv$, i.e., if and only if $v\in E\cap V(g^{-1}h,1)$.  Hence $Y=E\cap X$ for some $X\in \L_W$.  
As for the equivalences, clearly~\eqref{Prop:General:1} is equivalent to~\eqref{Prop:General:3}, which is equivalent to~\eqref{Prop:General:5} by 
Theorem~\ref{Thm:Lehrer:Springer}\eqref{LS:Hyperplanes}, and the remaining two equivalences \eqref{Prop:General:2}$\Leftrightarrow$\eqref{Prop:General:3} 
and \eqref{Prop:General:4}$\Leftrightarrow$\eqref{Prop:General:5} follow from Proposition~\ref{Prop:Springer:Transitive}\eqref{Springer:2}.
\end{proof}

\section{$G(m,p,n)$ case of Theorem~\ref{Thm:Main:Geometric}}\label{Section:Geometric}
Recall that $\bmu_m$ denotes the collection of all $m^{\text{th}}$ roots of unity, and that 
$G(m,p,n)$ denotes the group of all $n\times n$ monomial matrices with nonzero entries in $\bmu_m$ 
whose product lies in $\bmu_{m/p}$.
Note that $G(1,1,n)$ is the defining 
representation of $\mathfrak S_n$ given in \S\ref{Section:Introduction}, and that   
$G(1,1,n)\subseteq G(m,p,n)\subseteq G(m,1,n)$.
When $m>p$ the set of reflecting hyperplanes for $G(m,p,n)$ coincides with that for $G(m,1,n)$ and is given by the 
union of the following two sets:
\begin{align}
&\{z_i=\xi z_j\ :\ 1\leq i< j\leq n,\  \xi\in\bmu_m\}\label{Eq:Hyp:1}\\
&\{z_i=0\ :\ \hfill 1\leq i\leq n \}. \label{Eq:Hyp:2}
\end{align}
When $m=p$ the set of reflecting hyperplanes for $G(p,p,n)$ is simply given by~\eqref{Eq:Hyp:1}.

For roots of unity 
$\epsilonA_1,\epsilonA_2,\ldots, \epsilonA_\ell\in\bmu_m$ and an $\ell$-set $\{\sigma_1,\sigma_2,\ldots,\sigma_\ell\}\subseteq [n]$, 
identify the $2$-line array
\[\sigma=\begin{pmatrix} \sigma_1 & \sigma_2 &\cdots & \sigma_\ell \\ \epsilonA_1 \sigma_2 & \epsilonA_2 \sigma_3 & 
\cdots & \epsilonA_\ell \sigma_1\end{pmatrix}\]
with the linear map that 
fixes each $e_i$ with $i\in [n]\diff \{\sigma_1,\ldots,\sigma_\ell\}$ 
and that sends $e_{\sigma_i}$ to $\epsilonA_i e_{\sigma_{i+1}}$ for $i\in [\ell-1]$, while $e_{\sigma_\ell}\mapsto \epsilonA_\ell e_{\sigma_1}$.  
Because multiple arrays may represent the same map, in the next section 
we will require that $\sigma_1<\sigma_2,\sigma_3,\ldots,\sigma_\ell$, but we postpone the restriction until then.  
Call such an element $\sigma$ a \emph{(colored) cycle}, and define
\begin{itemize}
\item  $\ell(\sigma):=\ell$ (the \emph{length} of $\sigma$),
\item  $\Supp(\sigma):=\{\sigma_1,\ldots, \sigma_\ell\}$ (the \emph{support} of $\sigma$), and 
\item  $\Col(\sigma):=\{\epsilonA_1,\ldots, \epsilonA_\ell\}$ (the \emph{multiset of colors} of $\sigma$).
\end{itemize}
With two cycles 
$\sigma,\sigma'$ said to be \emph{disjoint} if $\Supp(\sigma)\cap\Supp(\sigma')=\varnothing$, note 
that any element of $G(m,1,n)$ may be decomposed as a product $\sigma^{(1)}\sigma^{(2)}\cdots\sigma^{(q)}$ 
of disjoint cycles, and that such a product is an element of $G(m,p,n)$ if and only if 
$\prod_i\prod_{\epsilonA\in\Col(\sigma^{(i)})}\epsilonA$ is an element of $\bmu_{m/p}$.  
The following lemma is a straightforward calculation.

\begin{lemma}\label{lemma:color}  
Let $\sigma\in G(m,1,n)$ be a cycle and write 
$\sigma=\begin{pmatrix} \sigma_1 & \sigma_2 &\cdots & \sigma_\ell \\ \epsilonA_1 \sigma_2 & \epsilonA_2 \sigma_3 & 
\cdots & \epsilonA_\ell \sigma_1\end{pmatrix}$.\\  Let $\zeta\neq 1$ be a root of unity.  Then
\[\dim V(\sigma,\zeta)=\begin{cases} 1 & \text{if }\zeta^\ell=\prod\epsilonA_i\text{;}\\ 0 & \text{otherwise.}\end{cases}\]
Moreover, in the former case $V(\sigma,\zeta)$ is the solution set of the following equations:
\begin{eqnarray}
&z_{\sigma_1}=\zeta \epsilonA_1^{-1}z_{\sigma_2}=\zeta^2(\epsilonA_1\epsilonA_2)^{-1}z_{\sigma_3}=\ldots =
\zeta^{\ell-1}(\epsilonA_1\epsilonA_2\cdots\epsilonA_{\ell-1})^{-1}z_{\sigma_\ell}\label{Eigen:Equation:1}\\
&z_i=0\quad\text{for}\quad i\in [n]\diff\{\sigma_1,\sigma_2,\ldots, \sigma_\ell\}.\label{Eigen:Equation:2}
\end{eqnarray}
\end{lemma}

The crux of Theorem~\ref{Thm:Main:Geometric} is Proposition~\ref{Prop:colored:ideal} below, for which we 
will need the following.

\begin{lemma}\label{Lemma:proper}
Let $W=G(m,p,n)$, let $g\in W$, and let $\zeta$ be a primitive $d^{\text{th}}$ root of unity.  
Suppose that $d\mid m$ and that $a(d)<n$.  
Then there exists an $i\in[n]$ such that for all ${\mathbf z}\in V(g,\zeta)$ one has $z_i=0$.
\end{lemma}

\begin{proof}
It suffices to assume that $V(g,\zeta)$ is maximal.  
Since $\zeta\in\bmu_m$, we have that $W$ contains 
\[h:=\begin{pmatrix} 1 \\ \zeta 1 \end{pmatrix}
\begin{pmatrix} 2 \\ \zeta 2\end{pmatrix}\cdots
\begin{pmatrix} a(d) \\ \zeta a(d) \end{pmatrix}
\begin{pmatrix} n \\ \zeta^{-a(d)}n\end{pmatrix},\]
which must have $\dim  V(h,\zeta)= a(d)$ by Proposition~\ref{Prop:Springer:Transitive}\eqref{Springer:1}, and therefore  
$z_n=0$ for each ${\mathbf z}\in V(h,\zeta)$.  
Since $W$ acts transitively on 
its maximal $\zeta$-eigenspaces by Proposition~\ref{Prop:Springer:Transitive}\eqref{Springer:2}, the result follows.
\end{proof}

In the next proposition we define another group $W'$ within the $G(m,p,n)$ family, that contains $W$.  
In particular, $\L_W\subset \L_{W'}$.  
We do so to obtain a stronger version of Theorem~\ref{Thm:Main:Geometric} (Theorem~\ref{Cor:Dowling} below) 
which will be used in 
Sections~\ref{Section:Consequences} 
and~\ref{Section:Type:A}.

\begin{proposition}\label{Prop:colored:ideal}
Let $W=G(m,p,n)$, let $\zeta$ be a primitive $d^{\text{th}}$ root of unity, 
and let $m\vee d$ denote the least common multiple of $m$ and $d$.  
Set
\[
W'
:=
\begin{cases} W & \text{if $a(d)=n$;}\\
G(m\vee d,1,n) & \text{if $a(d)<n$.}
\end{cases}
\]
Then $V(g,\zeta)\cap X\in E(W,\zeta)$ for every $g\in W$ and $X\in \L_{W'}$.
\end{proposition}

\begin{proof}
If $a(d)=n$, then one has $E(W,\zeta)=E(W,1)=\L_W$, and the result follows.

Assume that $a(d)<n$. 
It suffices to show that $V(g,\zeta)\cap H\in E(W,\zeta)$ for every $g\in W$ and every 
\emph{reflecting hyperplane} $H$ of $W'$, since a general element of $\L_{W'}$ is an intersection of 
reflecting hyperplanes.  
Let $g\in W$ and let $r\in W'$ be a reflection with fixed space $H$.
We show that $V(g,\zeta)\cap H$ is a $\zeta$-eigenspace of $W$ by exhibiting 
an $h\in W$ such that $V(h,\zeta)=V(g,\zeta)\cap H$.

Set $h=g$ if $V(g,\zeta)\subseteq H$.  Assume otherwise so that 
$V(g,\zeta)\not\subseteq H$.  Write
$g=\sigma^{(1)}\sigma^{(2)}\cdots \sigma^{(q)}$ as a maximal product of nonempty disjoint cycles so that 
$[n]=\cup_i\Supp(\sigma^{(i)})$ and $V(g,\zeta)=\oplus_i V(\sigma^{(i)},\zeta)$, and define 
$\Sigma:=\{\sigma^{(1)},\ldots, \sigma^{(q)}\}$.
Since $H$ is of the form $z_j=0$ or $z_j=\xi z_k$, we see from 
Lemma~\ref{lemma:color} that it contains all but exactly one or exactly two of the eigenspaces $V(\sigma^{(i)},\zeta)$.

\smallskip

\noindent{\sf Case 1.}  There exists exactly one cycle $\sigma\in\Sigma$ such that $V(\sigma,\zeta)\not\subseteq 
H$.

\noindent{\sf Subcase 1a.} $\ell(\sigma)\geq 2$.

Write $\sigma=\begin{pmatrix} \sigma_1 & \sigma_2 &\cdots & \sigma_{\ell-1} & \sigma_\ell \\ \epsilonA_1 \sigma_2 & \epsilonA_2 \sigma_3 & 
\cdots & \epsilonA_{\ell-1}\sigma_\ell &\epsilonA_\ell \sigma_1\end{pmatrix}$ and   
let $h$ be the element of $W$ that one obtains from $g$ by replacing $\sigma$ with the product of the cycles
\[
\sigma':=\begin{pmatrix} \sigma_1 & \sigma_2 &\cdots & \sigma_{\ell-2} & \sigma_{\ell-1} \\ \epsilonA_1 \sigma_2 & \epsilonA_2 \sigma_3 & 
\cdots & \epsilonA_{\ell-2}\sigma_{\ell-1} & \epsilonA_{\ell-1}\epsilonA_{\ell} \sigma_{1}\end{pmatrix}\quad\text{and}\quad 
\sigma'':=\begin{pmatrix}\sigma_{\ell}\\ \sigma_{\ell} \end{pmatrix}.
\]
Then $V(h,\zeta)=V(g,\zeta)\cap H$ if and only if $V(\sigma',\zeta)=V(\sigma'',\zeta)=\{0\}$.  
Clearly $V(\sigma'',\zeta)=\{0\}$, since $d>1$.  Suppose that $V(\sigma',\zeta)\neq\{0\}$.   Since 
$V(\sigma,\zeta)\neq\{0\}$ also, Lemma~\ref{lemma:color} tells us that
$\zeta^\ell=\epsilonA_1\epsilonA_2\cdots\epsilonA_\ell=\zeta^{\ell-1}$.    
However, $\zeta\neq 1$.

For example, if $W=G(2,2,8)$ and $\zeta=e^{2\pi i/6}$, then for 
\begin{align*}
g&=\begin{pmatrix} 1 &  2 & 3 \\ 2 & - 3 & 1\end{pmatrix}\begin{pmatrix} 5 & 6 & 7 \\ - 6 & - 7 & -5\end{pmatrix}
\begin{pmatrix} 4 & 8 \\ 8 & 4\end{pmatrix}\\
H&=\{{\mathbf z}\in\mathbb C^8\ :\ z_5=\zeta^2z_6\},
\end{align*}
one has $h=\begin{pmatrix} 1 &  2 & 3 \\ 2 & - 3 &  1\end{pmatrix}\left[\begin{pmatrix} 5 & 6  \\  -6 &  -(-5)\end{pmatrix}
\begin{pmatrix} 7 \\  7 \end{pmatrix}\right]
\begin{pmatrix} 4 & 8 \\ 8 & 4\end{pmatrix}$. 

\smallskip

\noindent{\sf Subcase 1b.} $\ell(\sigma)=1$.

Then $\sigma=\begin{pmatrix} \tau_0 \\ \zeta \tau_0\end{pmatrix}$ for some $\tau_0\in [n]$, implying 
that $\zeta\in\bmu_m$, and so $d\mid m$ and $a(d)<n$.  
Applying Lemmas~\ref{lemma:color} and~\ref{Lemma:proper}, it follows that $V(\tau,\zeta)=\{0\}$ for some $\tau\in \Sigma$.  
Note that necessarily $\tau\neq \sigma$, as their $\zeta$-eigenspaces disagree.  Write
\[\tau=\begin{pmatrix} \tau_1 & \tau_2 & \cdots & \tau_\ell \\ \deltaA_1 \tau_2 & \deltaA_2 \tau_3 & \cdots & \deltaA_\ell \tau_1\end{pmatrix}\]
and obtain $h$ from $g$ by replacing the two cycles $\sigma,\tau$ with the single cycle
\[\sigma' =\begin{pmatrix} \tau_0 & \tau_1 & \tau_2 & \cdots & \tau_\ell \\ \zeta \tau_1 & \deltaA_1 \tau_2 & \deltaA_2 \tau_3 & \cdots & \deltaA_\ell 
\tau_0\end{pmatrix}.\]
Applying Lemma~\ref{lemma:color} shows that $V(\sigma',\zeta)=\{0\}$; indeed, one has 
$\zeta^\ell\neq \prod_{i=1}^\ell\deltaA_i$, and so $\zeta^{\ell+1}\neq \zeta\prod_{i=1}^\ell\deltaA_i$.  
It follows that $V(h,\zeta)=V(g,\zeta)\cap H$.  That $h$ is in $W$ is clear.

For example, if $W=G(6,6,4)$ and $\zeta=e^{2\pi i/3}$, then for $g=
\begin{pmatrix} 2 \\ \zeta 2 \end{pmatrix}
\begin{pmatrix} 3 \\ \zeta 3 \end{pmatrix}
\begin{pmatrix} 4 \\ \zeta 4 \end{pmatrix}
\begin{pmatrix} 1 \\  1 \end{pmatrix}$ and $H=\{{\mathbf z}\in\mathbb C^4\ :\ z_3=0\}$, 
one has $h=\begin{pmatrix} 2 \\ \zeta 2 \end{pmatrix}
\begin{pmatrix} 4 \\ \zeta 4 \end{pmatrix}
\begin{pmatrix} 3 & 1 \\  \zeta 1 & 3 \end{pmatrix}$. \\ \\


\noindent{\sf Case 2.}  There exist exactly two cycles $\sigma,\tau\in\Sigma$ such that $V(\sigma,\zeta),V(\tau,\zeta)\not\subseteq H$.

Then for some $\xi\in \bmu_{m\vee d}$ and a suitable indexing, we have
\[
\sigma=\begin{pmatrix} \sigma_1 & \sigma_2 & \cdots & \sigma_s\\ 
\epsilonA_1 \sigma_2 & \epsilonA_2 \sigma_3 & \cdots & \epsilonA_s \sigma_1\end{pmatrix}, 
\tau=\begin{pmatrix}\tau_1 & \tau_2 & \cdots & \tau_t \\ 
\deltaA_1\tau_2 & \deltaA_2\tau_3 & \cdots & \deltaA_t\tau_1\end{pmatrix},
r=\begin{pmatrix} \sigma_1 & \tau_1 \\ \xi^{-1} \tau_1 & \xi \sigma_1\end{pmatrix}
\]
Since $H$ is given by $z_{\sigma_1}=\xi z_{\tau_1}$, Lemma~\ref{lemma:color} implies that 
$\left(V(\sigma,\zeta)\oplus V(\tau,\zeta)\right)\cap H$ 
consists of the points $\mathbf{z}\in\mathbb C^n$ that satisfy 
$z_i=0$ for $i\not\in\Supp(\sigma)\cup\Supp(\tau)$ and 
\begin{equation}
\begin{split}
z_{\sigma_1}=\zeta \epsilonA_1^{-1}z_{\sigma_2}=
\ldots = \zeta^{s-1}&(\epsilonA_1\cdots\epsilonA_{s-1})^{-1}z_{\sigma_s}\\
=\,&\xi z_{\tau_1}=\xi \zeta \deltaA_1^{-1}z_{\tau_2}=
\ldots = \xi\zeta^{t-1}(\deltaA_1\cdots\deltaA_{t-1})^{-1}z_{\tau_t}.
\end{split}\label{eq:merge}
\end{equation}
Let $k$ be such that the coefficient $\deltaA$ of $z_{\tau_k}$ in~\eqref{eq:merge} is an element of $\bmu_m$.  
(For existence, note that the cosets $\bmu_m,\zeta\bmu_m,\ldots,\zeta^{t-1}\bmu_m$ cover 
the group $\bmu_d\bmu_m$, since $\zeta$ generates $\bmu_d$ and $\zeta^t\in\bmu_m$, then observe 
that $\xi$ permutes these cosets, since 
$\bmu_{m\vee d}=\bmu_m\bmu_d$ from basic algebra.)  Then
\begin{equation}
\left(V(\sigma,\zeta)\oplus V(\tau,\zeta)\right)\cap H=\left(V(\sigma,\zeta)\oplus V(\tau,\zeta)\right)\cap H'\label{Eq:Intersection}
\end{equation}
for $H'$ the reflecting hyperplane of $r'=\begin{pmatrix} \sigma_1 & \tau_k\\ \deltaA^{-1} \tau_k & \deltaA\sigma_1\end{pmatrix}$.

We claim that $h:=gr'$ satisfies $V(h,\zeta)=V(g,\zeta)\cap H$, or in other words, 
that $V(\sigma\tau r')=V(\sigma\tau)\cap H$.  
To see this, employ~\eqref{Eq:Intersection} to rewrite the equality as
\begin{equation}
V(\sigma\tau r',\zeta)=\left(V(\sigma,\zeta)\oplus V(\tau,\zeta)\right)\cap H'.\label{Eq:Claim}
\end{equation}
Now observe that, on one hand, 
$V(\sigma\tau,\zeta)\cap H'$ is clearly contained in $V(\sigma\tau r',\zeta)$ 
and has dimension $1$ by hypothesis.  On the other hand, 
$\sigma\tau r'$ is necessarily a cycle, and therefore has $\zeta$-eigenspace of dimension at most $1$ by Lemma~\ref{lemma:color}. 

For example, if $W=G(6,3,8)$ and $\zeta=e^{2\pi i/9}$, then for $\omega:=e^{2\pi i/6}$ and 
\begin{align*}
g&=\begin{pmatrix} 1 & 2 & 3 \\ \omega 2 & 3 & \omega 1 \end{pmatrix} \begin{pmatrix} 4 & 5 & 6 \\
 - 5 & \omega^4 6 & \omega 4\end{pmatrix}
\begin{pmatrix}7 & 8 \\ \omega 8 & \omega 7\end{pmatrix}\\
\xi&= e^{2\pi i / 18}\\
r&=\begin{pmatrix} 1 & 4 \\ \xi^{-1} 4 & \xi 1\end{pmatrix}, 
\end{align*}
one has $r'=\begin{pmatrix} 1 & 6 \\ \omega^2 6 & \omega^{-2} 1 \end{pmatrix}$,\ \   
$h=\begin{pmatrix} 1 & 4 & 5 & 6 & 2 & 3 \\ -4 & -5 & \omega^4 6 & \omega^52 & 3 & \omega 1 \end{pmatrix}
\begin{pmatrix} 7 & 8 \\ \omega 8 & \omega 7\end{pmatrix}$.
\end{proof}

\begin{proof}[{\bf Proof of Theorem~\ref{Thm:Main:Geometric} for $\mathbf{G(m,p,n)}$}]
Note that $\L_W\subseteq \L_{W'}$ in Proposition~\ref{Prop:colored:ideal}, since $W$ is a subgroup of $W'$, 
then invoke Proposition~\ref{Prop:Maximal}.
\end{proof}

\noindent
Another consequence of Proposition~\ref{Prop:colored:ideal} is the following stronger result, which 
will play an important role in Sections~\ref{Section:Consequences} and~\ref{Section:Type:A} below.

\begin{theorem}\label{Cor:Dowling}Let $W=G(m,p,n)$, let $\zeta$ be a primitive $d^{\text{th}}$ root of unity, and let $W'$ be as in 
Proposition~\ref{Prop:colored:ideal}.  
Then
\begin{enumerate}[(i)]
\item $E(W,\zeta)\subseteq \L_{W'}$, and 
\item $V(g,\zeta)\cap X\in E(W,\zeta)$ for every $g\in W$ and $X\in \L_{W'}$.
\end{enumerate}
In other words, $E(W,\zeta)$ is an upper order ideal of $\L_{W'}$.
\end{theorem}

\begin{proof}
Observe that $W'$ contains the scalar matrix $\zeta$, and that 
\[V(g,\zeta)=V(\zeta^{-1}g,1)\in \L_{W'}\] whenever $g\in W$.  The first claim follows, and 
the second is Proposition~\ref{Prop:colored:ideal}.
\end{proof}

\section{Maximal eigenspaces of $G(m,p,n)$}\label{Section:Maximal:Eigenspaces}  
In this section we associate a certain \emph{word}, denoted $\word(E)$, to each maximal 
eigenspace $E\in E(W,\zeta)$ for $W=G(m,p,n)$ and $\zeta\neq 1$.  
In the next section we show that lexicographically ordering these words gives 
a recursive atom ordering for $\widehat{E(W,\zeta)}$.  In addition to $a(d)$, the 
following number plays an important role in our discussion.
\begin{align}
\ell(d)&\Def \min\{s\in\mathbb Z_{\geq 1}\ :\ \zeta^s\in\bmu_m\}\\
&=\frac{d}{\gcd(m,d)}.
\end{align}
The crux of our construction is that each maximal eigenspace $E\in E(W,\zeta)$ determines 
a unique set of $a(d)$ many $\ell(d)$-cycles in $G(m,1,n)$ whose product 
has $\zeta$-eigenspace $E$.  It is from   
this set that we construct $\word(E)$ in Corollary~\ref{Cor:Word} below.  We establish the correspondence by 
first showing that any product $g\in G(m,1,n)$ of $\dim V(g,\zeta)$ many nontrivial $\ell(d)$-cycles 
is uniquely determined by its eigenspace $V(g,\zeta)$, and then showing that 
each maximal $E\in E(W,\zeta)$ may be realized as the $\zeta$-eigenspace of 
such a product.

Recall from Section~\ref{Section:Geometric} the identification of a 2-line array of the form 
\begin{equation}
\sigma=\begin{pmatrix} 
\sigma_1 & \sigma_2 & \cdots & \sigma_\ell\\ 
\epsilonA_1 \sigma_2 & \epsilonA_2\sigma_3 & \cdots & \epsilonA_\ell \sigma_1
\end{pmatrix}\label{Generic:Permutation}
\end{equation}
and a particular element of $G(m,1,n)$, and note 
that the element determines the 2-line array up to cyclically permuting columns.  Thus, by requiring 
that the smallest $\sigma_i$ come first, the array is uniquely determined.  We adopt this convention 
for the remainder of the section, i.e., that $\sigma_1<\sigma_2,\sigma_3,\ldots,\sigma_\ell$.

\begin{lemma}\label{Lemma:Eigenspace:Equality}
Let $\zeta$ be a primitive $d^{\text{th}}$ root of unity for $d>1$, and set $\ell:=\ell(d)$.  Suppose that 
$\sigma,\tau\in G(m,1,n)$ are two $\ell$-cycles such that $V(\sigma,\zeta)=V(\tau,\zeta)\neq\{0\}$.  
Then $\sigma=\tau$.
\end{lemma}

\begin{proof}
We show that the map $\sigma\mapsto V(\sigma,\zeta)$ is a bijection when restricted 
to the $\ell$-cycles $\sigma$ such that $V(\sigma,\zeta)\neq 0$ by 
constructing its inverse.  Fix such a cycle $\sigma$ and 
label its entries as in~\eqref{Generic:Permutation} so 
that its image $V(\sigma,\zeta)$ is defined by 
equations~\eqref{Eigen:Equation:1} and~\eqref{Eigen:Equation:2} of Lemma~\ref{lemma:color}.  
Working backwards, first note that $V(\sigma,\zeta)$ is a line, and therefore uniquely 
determines~\eqref{Eigen:Equation:1} and~\eqref{Eigen:Equation:2}.
Next observe that
$\bmu_{m\vee d}$ is the disjoint union of 
the cosets $\bmu_m,\, \zeta\bmu_m,\, \ldots,\,  \zeta^{\ell}\bmu_m$, so that
for each $i$ there is a \emph{unique} 
scalar $\zeta^j(\delta_1\delta_2\cdots\delta_j)^{-1}$ appearing in~\eqref{Eigen:Equation:1} that is
contained in the coset $\zeta^i\bmu_m$, from which one recovers
$\delta_1\delta_2\cdots\delta_{i-1}$ and $\sigma_i$.   
The cycle $\sigma$ is obtained by letting $i$ range 
from $1$ to $\ell$ while taking successive quotients so as to isolate each $\delta_i$.
\end{proof}

The general case follows:

\begin{proposition}\label{Prop:Unique:Sets}
Let $\zeta$ and $\ell$ be as in Lemma~\ref{Lemma:Eigenspace:Equality}.  Suppose that 
\[S:=\{\sigma^{(1)},\sigma^{(2)},\ldots,\sigma^{(q)}\}\quad\text{ and }\quad T:=\{\tau^{(1)},\tau^{(2)},\ldots,\tau^{(q)}\}\]
are two sets of pairwise disjoint $\ell$-cycles in $G(m,1,n)$ that satisfy
\begin{enumerate}[(1)]
\item $\dim V(\sigma^{(i)},\zeta)=\dim V(\tau^{(i)},\zeta)=1\text{\ \ for all\ \ }i\in[q]\text{, and}$
\item $\bigoplus_i V(\sigma^{(i)},\zeta)
=
\bigoplus_i V(\tau^{(i)},\zeta)$.
\end{enumerate}
Then $S=T$.
\end{proposition}

\begin{proof}
Consider the image of $\bigoplus_i V(\sigma^{(i)},\zeta)$ under the orthogonal projection of $\mathbb C^n$ onto 
$\bigoplus_{t\in \Supp(\tau^{(1)})}\mathbb C e_t$ and employ Lemma~\ref{lemma:color} to show that 
$V(\tau^{(1)},\zeta)$ is equal to $V(\sigma^{(k)},\zeta)$ for some $k$, then apply Lemma~\ref{Lemma:Eigenspace:Equality}.  
The result follows by induction.
\end{proof}

We now come to the crux of this section.

\begin{proposition}\label{Prop:Smallest}
Let $\zeta$ and $\ell$ be as in Lemma~\ref{Lemma:Eigenspace:Equality}, and set $W=G(m,p,n)$.  
Let $E\in E(W,\zeta)$ be maximal under inclusion.  Then there exists a set 
of pairwise disjoint $\ell$-cycles $\{\sigma^{(1)},\ldots, \sigma^{(a(d))}\}\subset G(m,1,n)$ 
such that \[E=V(\sigma^{(1)},\zeta)\oplus\cdots\oplus V(\sigma^{(a(d))},\zeta).\]
\end{proposition}

\begin{proof}
Choose $g\in W$ such that $E=V(g,\zeta)$, and write $g=\tau^{(1)}\tau^{(2)}\cdots\tau^{(q)}$ as a product of 
disjoint cycles, indexed so that 
\[
\dim V(\tau^{(i)},\zeta)
=
\begin{cases}
1 & \text{if $1\leq i\leq a(d)$;}\\
0 & \text{otherwise.}
\end{cases}
\]
Lemma~\ref{lemma:color} tells us that $\ell(\tau^{(i)})\geq \ell $ for each $i\in [a(d)]$, and so $n\geq a(d)\ell$.  
If $n=a(d)\ell $, the claim follows.  Assume otherwise so that $n>a(d)\ell$, and set
\[h=\prod_{i=0}^{a(d)-1}
\begin{pmatrix}
i\ell+1 & i\ell+2 & \cdots & i\ell+\ell\\
\zeta^\ell(i\ell+2) & i\ell+3 & \cdots & i\ell+1
\end{pmatrix}.
\]
Let $\pi'$ be the element of $G(m,1,n)$ mapping $e_n$ to $\zeta^{-a(d)\ell}e_n$ while fixing all other $e_i$.  
Clearly $h\pi'\in G(m,p,n)$ and 
\[V(h\pi',\zeta)=V(\pi^{(0)},\zeta)\oplus V(\pi^{(1)},\zeta)\oplus \cdots\oplus V(\pi^{(a(d)-1)},\zeta)\oplus V(\pi',\zeta)\]
for \[\pi^{(i)}:=\begin{pmatrix}  i\ell+1 & i\ell+2 &\cdots & i\ell+\ell  \\ 
\zeta^\ell(i\ell+2) & i\ell+3 & \cdots & i\ell+1\end{pmatrix}.\]  
Since Lemma~\ref{lemma:color} implies $\dim V(\pi^{(i)},\zeta)=1$ for 
$0\leq i\leq a(d)-1$, it follows that 
$V(h\pi',\zeta)$ is a maximal eigenspace in $E(W,\zeta)$.  
Hence the result, as $W$ acts transitively on its maximal $\zeta$-eigenspaces by Proposition~\ref{Prop:Springer:Transitive}\eqref{Springer:2}.
\end{proof}

We can now label each maximal eigenspace of $G(m,p,n)$ by the cycles that it determines.  
In the next section we shall use this labeling of the maximal eigenspaces to construct a 
CL-shelling of $\widehat{E(W,\zeta)}$.

\begin{corollary}\label{Cor:Word}
Let $\zeta$ and $\ell$ be as in Lemma~\ref{Lemma:Eigenspace:Equality}, and set $W=G(m,p,n)$.  
Let $E\in E(W,\zeta)$ be a maximal eigenspace under inclusion.  
Then there exists a unique sequence 
\[\word(E):=(s_1,s_2,\ldots,s_{n-a(d)\ell}\ ;\ \sigma^{(1)},\ldots, \sigma^{(a(d))})\] 
of integers $s_i$ and disjoint $\ell$-cycles $\sigma^{(i)}\in G(m,1,n)$ with the following properties.
\begin{enumerate}[(i)]
\item  $E=\bigoplus_i V(\sigma^{(i)},\zeta)$.
\item  $\{s_1,s_2,\ldots,s_{n-a(d)\ell}\}=[n]\diff \bigcup_{i=1}^{a(d)}\Supp(\sigma^{(i)})$.
\item  $s_1<s_2<\ldots< s_{n-a(d)\ell}$.
\item  $\min\Supp (\sigma^{(i)})<\min \Supp(\sigma^{(j)})$ whenever $i<j$.
\end{enumerate}
\end{corollary}

\begin{definition}\label{Def:Lex:Order}
Linearly order $\bmu_m$ as follows: set $\epsilonA:=e^{2\pi i/m}$ and define $\epsilonA^j<\epsilonA^k$ whenever $0\leq j<k< m$.  
Suppose that $\sigma$ and $\tau$ are two distinct cycles of the same length in $G(m,1,n)$, 
and let $k$ index the first column in which they differ:
\begin{align*}
\sigma&=\begin{pmatrix} 
\sigma_1 & \sigma_2 & \cdots & \sigma_k & \cdots & \sigma_\ell\\ 
\epsilonA_1\sigma_2 & \epsilonA_2\sigma_3 & \cdots & \epsilonA_k\sigma_{k+1} & \cdots &\epsilonA_\ell \sigma_1
\end{pmatrix}\\
\tau&=\begin{pmatrix} 
\tau_1 & \tau_2 & \,\cdots & \tau_k & \hspace{.5mm}\cdots & \tau_\ell\hspace{.35mm}\\ 
\deltaA_1\tau_2 & \deltaA_2\tau_3 & \,\cdots & \deltaA_k\tau_{k+1} & \hspace{.5mm}\cdots & \deltaA_\ell \tau_1\hspace{.35mm}
\end{pmatrix}.
\end{align*}
Define $\sigma<_\lex\tau$ if the $k$th column of $\sigma$ is lexicographically less than that of $\tau$ in the sense that 
one of the following holds.
\begin{enumerate}[(a)]
\item  $\sigma_k<\tau_k$.
\item  $\sigma_k=\tau_k$ and $\sigma_{k+1}<\tau_{k+1}$.
\item  $\sigma_k=\tau_k$, $\sigma_{k+1}=\tau_{k+1}$, and $\epsilonA_k<\deltaA_k$.
\end{enumerate}
\end{definition}

\begin{definition}
Let $W=G(m,p,n)$ and let $\zeta$ be a primitive $d^{\text{th}}$ root of unity for $d>1$.  
For two words 
\begin{align*}
\word(E)&=(s_1,\ldots, s_{n-a(d)\ell}\ ;\ \sigma^{(1)},\ldots,\sigma^{(a(d))})\\
\word(E')&=(t_1,\ldots, t_{n-a(d)\ell}\ ;\ \tau^{(1)},\ldots,\tau^{(a(d))})
\end{align*}
of distinct maximal eigenspaces $E,E'\in E(W,\zeta)$,
define $\word(E)<_\lex \word(E')$ if in the first position in which they differ, the 
term of $\word(E)$ is strictly less than the corresponding term of $\word(E')$.
\end{definition}

\begin{example}  
Ordering the three maximal $(-1)$-eigenspaces $E_i$ of $\mathfrak S_4$ by their words $\word(E_i)$, we have 
$E_1<_\lex E_2<_\lex E_3$ for $E_i$ the eigenspace labeled by \circled{$i$} in Figure~\ref{Figure:A2}; see
Table~\ref{Table:Dimensions}.  (Note that $n=a(d)\ell$ in this case.)
\begin{table}[hbt]
\center
\begin{tabular}{c@{\hskip 1.1cm}c@{\hskip 1.1cm}c}
\toprule
$i$ & $E_i=\{\mathbf z\in\mathbb C\ \text{satisfying }\ldots\}$  & $\word(E_i)$ \\ 
\midrule
& & \\[-12pt]
1 & $z_1=-z_2$, $z_3=- z_4$ & $\left(\begin{pmatrix} 1 & 2 \\ 2 & 1 \end{pmatrix},\begin{pmatrix} 3 & 4 \\ 4 & 3\end{pmatrix}\right)$\\[9pt]
2 & $z_1=-z_2$, $z_3=- z_4$ & $\left(\begin{pmatrix} 1 & 3 \\ 3 & 1 \end{pmatrix},\begin{pmatrix} 2 & 4 \\ 4 & 2\end{pmatrix}\right)$\\[9pt]
3 & $z_1=-z_2$, $z_3=- z_4$ & $\left(\begin{pmatrix} 1 & 4 \\ 4 & 1 \end{pmatrix},\begin{pmatrix} 2 & 3 \\ 3 & 2\end{pmatrix}\right)$\\[9pt]
\bottomrule
\end{tabular}
\tableskip
\caption{The maximal $(-1)$-eigenspaces $E_i$ of $\mathfrak S_4$ in Figure~\ref{Figure:A2}, indexed with respect 
to lexicographic order on their words.}\label{Table:Dimensions}
\end{table}

\noindent
The maximal spaces in Figures~\ref{Figure:A1}-\ref{Figure:B} are similarly indexed.  In the poset 
$E(\mathfrak S_4,\zeta)$ of Figure~\ref{Figure:A3}, for example, the eigenspaces
\begin{align*}
E_3&=\{\mathbf z\in\mathbb C\ :\ z_2=0,\ z_1=\zeta z_3=\zeta^2 z_4\}\\
E_5&=\{\mathbf z\in\mathbb C\ :\ z_3=0,\ z_1=\zeta z_2=\zeta^2 z_4\}
\end{align*}
have words
\begin{align*}
\word(E_3)&=\left(2,\begin{pmatrix} 1 & 3 & 4\\ 3 & 4 & 1 \end{pmatrix}\right)\\
\word(E_5)&=\left(3,\begin{pmatrix} 1 & 2 & 4\\ 2 & 4 & 1 \end{pmatrix}\right)
\end{align*}
such that $\word(E_3)<_\lex\word(E_5)$.
\end{example}

\section{$G(m,p,n)$ case of Theorem~\ref{Thm:Main:Shelling}}\label{Section:Shelling}
For $P$ a finite graded poset, denote its \emph{rank function} by $r(x):P\to \mathbb Z$ (with minimal elements having rank $0$),
and its \emph{rank} by $r(P):=\max\{r(x)\ :\ x\in P\}$.  
Recall that a poset is \emph{bounded} if it contains both a bottom element $\hat{0}$ and a top element $\hat{1}$, and that an 
\emph{atom} in a poset with a $\hat{0}$ is any element that covers $\hat{0}$.

\begin{definition}[Bj\"orner-Wachs~\cite{BW}]\label{RAO}
A bounded poset $P$ is said to \emph{admit a recursive atom ordering} if its rank $r(P)$ is $1$, or if 
$r(P)>1$ and there is an ordering of the atoms $a_1,\ldots, a_t$ that satisfies the following.
\begin{enumerate}[(i)]
\item\label{RAO:1}  Each interval $[a_j,\hat{1}]$ admits a recursive atom ordering in which its atoms that are contained in 
$[a_i,\hat{1}]$ for some $i<j$ come first. 
\item\label{RAO:2}  If $i<j$ and $a_i,a_j<x$, then there exists a $k<j$ and an atom $\tilde{x}$ of $[a_j,\hat{1}]$ for 
which $a_k<\tilde{x}\leq x$.
\end{enumerate}
\end{definition}

A well-known result of Bj\"orner and Wachs~\cite{BW} states 
that any ordering of the atoms in a \emph{totally semimodular poset} is a recursive atom ordering.  
In particular, any ordering of the atoms in a semimodular lattice is a recursive atom ordering, 
from which the next useful lemma follows immediately.

\begin{lemma}[Lemma 3 in~\cite{Sagan}]\label{Sagan:Lemma}
If $P$ is a bounded poset in which $[a,\hat{1}]$ is a semimodular lattice for every atom $a\in P$, 
then an atom ordering $a_1,\ldots, a_t$ is 
a recursive atom ordering if and only if it satisfies condition~\eqref{RAO:2} of Definition~\ref{RAO}.
\end{lemma}

Our goal is to give, when $W=G(m,p,n)$ and $\zeta$ is a root of unity, 
a recursive atom ordering for $P=\widehat{E(W,\zeta)}$, whose atoms 
are the maximal $\zeta$-eigenspaces of $W$ under inclusion.  By Lemma~\ref{Sagan:Lemma}, 
this amounts to producing a candidate ordering and verifying condition~\eqref{RAO:2} of Definition~\ref{RAO}.  
The case $d=1$ is \emph{trivial} in the sense that there is only one atom.
For $d>1$ we order the atoms by their words:

\begin{theorem}\label{Theorem:Shelling}
Let $W=G(m,p,n)$ and let $\zeta$ be a primitive $d^{\text{th}}$ root of unity for $d>1$.  
Then the lexicographic ordering of atoms $E$ of $\widehat{E(W,\zeta)}$ by their words 
$\word(E)$ is a recursive atom ordering.
\end{theorem}

\begin{proof}
By Theorem~\ref{Thm:Main:Geometric} and Lemma~\ref{Sagan:Lemma}, we need only 
verify that the atom ordering satisfies condition~\eqref{RAO:2} of Definition~\ref{RAO}.  
The result follows immediately if there is only one atom, so assume otherwise and   
note that $a(d)< n$ by Corollary~\ref{Cor:Springer}.

Suppose that $A,B\in E(W,\zeta)$ are two atoms with $\word(A)<_\lex \word(B)$. Write 
\begin{align*}
\word(A)&=(a_1,\ldots, a_{n-a(d)\ell}\ ;\ \sigma^{(1)},\ldots,\sigma^{(a(d))})\\
\word(B)&=(b_1,\ldots, b_{n-a(d)\ell}\ ;\ \tau^{(1)},\ldots,\tau^{(a(d))}).
\end{align*}
Since any element that lies above both $A$ and $B$ in $E(W,\zeta)$ must be a subspace of $A\cap B$,   
it suffices to exhibit a maximal eigenspace $C\in E(W,\zeta)$ that satisfies 
\begin{enumerate}[(i)]
\item $\word(C)<_\lex \word(B)$, and \label{Want:1}
\item $A\cap B\subseteq B\cap H \subseteq C$ for some hyperplane $H\in \L_{W'}$,\label{Want:2}
\end{enumerate}
where $W'$ is as in Proposition~\ref{Prop:colored:ideal}.  

Choose (possibly empty) cycles $\sigma^{(0)},\tau^{(0)}\in G(m,1,n)$ such that 
\begin{align*}
g&:=\sigma^{(0)}\sigma^{(1)}\cdots\sigma^{(a(d))}\\
h&:=\tau^{(0)}\tau^{(1)}\cdots\tau^{(a(d))}
\end{align*} 
are products of disjoint cycles with $g,h\in W$.  Then $A=V(g,\zeta)$ and $B=V(h,\zeta)$, 
since $A=\bigoplus_{i\geq 1}V(\sigma^{(i)},\zeta)$ and 
$B=\bigoplus_{i\geq 1}V(\tau^{(i)},\zeta)$ are maximal.

\ \\
\noindent{\sf Case 1.}  $a_i\neq b_i$ for some $i$. 

Let $i$ be the smallest such index.  Then $a_i< b_i$ and $a_i\in \Supp(\tau^{(j)})$ for some $j\geq 1$.  
Set $C:= r B=V(rhr^{-1},\zeta)$ for \[r=\begin{pmatrix}a_i & b_i \\ b_i & a_i\end{pmatrix}.\]  
Since one obtains $\word(C)$ by interchanging $a_i$ and $b_i$ in $\word(B)$, it follows that   
$\word(C)<_\lex \word(B)$.  Applying $r$ to both sides of 
$B\cap H_r\subseteq B$ shows that $B\cap H_r\subseteq C$.  Lastly, $A\cap B\subseteq B\cap H_r$ follows from 
the fact that $z_{a_i}=0$ on $A$ and $z_{b_i}=0$ on $B$.

For example, if $W=G(4,2,8)$ and $\zeta=e^{2\pi i/6}$, then for $\omega:=e^{2\pi i/4}$ and
\begin{align*}
\word(A)&=\left(4,7,\begin{pmatrix} 1 & 3 & 2 \\ \omega 3 & - 2 & \omega 1 \end{pmatrix}, 
\begin{pmatrix} 5 & 6 & 8 \\ - 6 & - 8 & 5\end{pmatrix}\right)\\
\word(B)&=\left(4,8,\begin{pmatrix} 1 & 6 & 5\\ - 6 & 5 & - 1\end{pmatrix},
\begin{pmatrix} 2 & 7 & 3 \\ 7 & 3 & 2\end{pmatrix}\right),
\intertext{one has $r=\begin{pmatrix} 7 & 8 \\ 8 & 7\end{pmatrix}$ and}
\word(C)&=\left(4,7,\begin{pmatrix} 1 & 6 & 5\\ - 6 & 5 & - 1\end{pmatrix},
\begin{pmatrix} 2 & 8 & 3 \\ 8 & 3 & 2\end{pmatrix}\right).
\end{align*}

\ \\
\noindent{\sf Case 2.}  $a_i=b_i$ for all $i$.  

Let $j\geq 1$ be the smallest integer for which $\sigma^{(j)}\neq \tau^{(j)}$.  Set $\sigma:=\sigma^{(j)}$ and 
$\tau:=\tau^{(j)}$, and let $k$ index the first column in which $\sigma$ and $\tau$ differ:
\begin{align*}
\sigma&=\begin{pmatrix} 
\sigma_1 & \sigma_2 & \cdots & \sigma_k & \cdots & \sigma_\ell\\ 
\epsilonA_1\sigma_2 & \epsilonA_2\sigma_3 & \cdots & \epsilonA_k\sigma_{k+1} & \cdots &\epsilonA_\ell \sigma_1
\end{pmatrix}\\
\tau&=\begin{pmatrix} 
\tau_1 & \tau_2 & \,\cdots & \tau_k & \hspace{.5mm}\cdots & \tau_\ell\hspace{.35mm}\\ 
\deltaA_1\tau_2 & \deltaA_2\tau_3 & \,\cdots & \deltaA_k\tau_{k+1} & \hspace{.5mm}\cdots & \deltaA_\ell \tau_1\hspace{.35mm}
\end{pmatrix}.
\end{align*}
Note that $\sigma_i=\tau_i$ for $i\leq k$, and $\epsilonA_i=\deltaA_i$ for $i<k$.   

By Lemma~\ref{lemma:color}, each ${\mathbf{z}}\in A\cap B$ satisfies the two equations
\begin{align*}
z_{\sigma_1}&=\zeta^k\epsilonA_1^{-1}\cdots\epsilonA_k^{-1}z_{\sigma_{k+1}}\\
z_{\tau_1}&=\zeta^k\deltaA_1^{-1}\cdots\deltaA_k^{-1}z_{\tau_{k+1}}.
\end{align*}  
Since $\sigma_1=\tau_1$ and $\epsilonA_i=\deltaA_i$ for $i<k$, it follows that 
\begin{equation}
z_{\sigma_{k+1}}=\epsilonA_k\deltaA_k^{-1} z_{\tau_{k+1}}\quad\text{for all}\quad{\mathbf z}\in A\cap B.
\label{Eq:Shell}\end{equation}  

\noindent{\sf Case 2a.}  $\sigma_{k+1}<\tau_{k+1}$. 

Equation~\eqref{Eq:Shell} says that  
  $A\cap B\subseteq H_r$ for the reflection   
\[r=\begin{pmatrix}\sigma_{k+1} & \tau_{k+1} \\ \deltaA_k\epsilonA_k^{-1}\tau_{k+1} & \epsilonA_k\deltaA_k^{-1}\sigma_{k+1}\end{pmatrix}.\]
Set  $C:=r B=V(rhr^{-1},\zeta)$ so that, as in Case 1, one has $A\cap B\subseteq B\cap H_r\subseteq C$.
Then $\word(C)<_\lex \word(B)$, since $\sigma_{k+1}$ occurs to the right of 
$\tau_{k+1}$ in $\word(B)$.  (More precisely, either $\sigma_{k+1}=\tau_{k+1+i}$ 
for some $i\geq 1$, or 
$\sigma_{k+1}$ is in the support of $\tau^{(l)}$ 
for some $l>j$.) 

For example, if $W=G(4,2,9)$ and $\zeta=e^{2\pi i/16}$, then for $\omega:=e^{2\pi i/4}$ and
\begin{align*}
\word(A)&=\left(4,\begin{pmatrix} 1 & 3 & 2 & 7 \\  \omega 3 & - 2 &  - 7 & 1 \end{pmatrix}, 
\begin{pmatrix} 5 & 9 & 6 & 8 \\ 9 & \omega 6 & 8 & 5 \end{pmatrix}\right)\\
\word(B)&=\left(4,\begin{pmatrix} 1 & 3 & 9 & 8 \\  \omega 3 & \omega^{-1} 9 & 8 & \omega 1 \end{pmatrix},
\begin{pmatrix} 2 & 6 & 5 & 7 \\ \omega^{-1} 6 & 5 & - 7 & 2 \end{pmatrix}\right),
\intertext{one has $k=2$, $r=\begin{pmatrix} 2 & 9 \\  \omega 9 &  \omega^{-1} 2\end{pmatrix}$, and}
\word(C)&=\left(4,\begin{pmatrix} 1 & 3 & 2 & 8\\ \omega 3 &- 2 & \omega 8& \omega 1\end{pmatrix},
\begin{pmatrix}  5 & 7 & 9 & 6 \\ - 7& \omega 9 & - 6 & 5 \end{pmatrix}\right).
\end{align*}
\noindent{\sf Case 2b.}  $\sigma_{k+1}=\tau_{k+1}$. 

Then $\epsilonA_k<\deltaA_k$.  
It follows from~\eqref{Eq:Shell} that $z_{\tau_{k+1}}=0$, and hence $z_i=0$ whenever $i\in \Supp(\tau)$, 
for each ${\mathbf{z}}\in A\cap B$.
In particular, $A\cap B\subseteq B\cap H_r=B\cap H_s$ for  
\[r=\begin{pmatrix} \tau_k \\ \epsilonA_k\deltaA_k^{-1}\tau_k\end{pmatrix}\quad\text{and}\quad 
s=\begin{pmatrix}\tau_\ell \\ \epsilonA_k^{-1}\deltaA_k\tau_\ell\end{pmatrix}.\]  
Set $h':=h r s \in W$ and $C:=V(h',\zeta)$.  First note that $C$ is maximal, since 
$B$ is maximal and the cycle $\tau r s$ in $h'$ has $\dim V(\tau r s, \zeta)=1$ by 
Lemma~\ref{lemma:color}.  
It is also clear that $B\cap H_r=C\cap H_r$, and hence 
$A\cap B\subseteq B\cap H_r\subseteq C$.  It remains to see that $C\neq B$, from which it follows 
that $\word(C)<_\lex \word(B)$.  To this end, observe that $k<\ell$, since 
since  
$\epsilonA_i=\deltaA_i$ for $i<k$ and $\epsilonA_1\cdots\epsilonA_\ell=\zeta^\ell=\deltaA_1\cdots\deltaA_\ell$, while $\epsilonA_k<\deltaA_k$.  
It follows that $\tau rs\neq \tau$, and therefore $C\neq B$ by Proposition~\ref{Prop:Unique:Sets}.

For example, if $W=G(4,2,8)$ and $\zeta=e^{2\pi i/6}$, then for $\omega:=e^{2\pi i/4}$ and
\begin{align*}
\word(A)&=\left(4,7,\begin{pmatrix} 1 & 3 & 2 \\ \omega 3 & -2 & \omega 1 \end{pmatrix}, \begin{pmatrix} 5 & 6 & 8 \\ -6 & -8 & 5\end{pmatrix}\right)\\
\word(B)&=\left(4,7,\begin{pmatrix} 1 & 3 & 2\\ - 3 & 2 & -1\end{pmatrix},\begin{pmatrix} 5 & 8 & 6 \\ 8 & 6 & 5\end{pmatrix}\right),
\intertext{one has $k=1$, $r=\begin{pmatrix} 1 \\ \omega^{-1} 1 \end{pmatrix}$, $s=\begin{pmatrix} 2 \\ \omega 2 \end{pmatrix}$, and}
\word(C)&=\left(4,7,\begin{pmatrix} 1 & 3 & 2\\ \omega 3 &  2 & \omega^{-1} 1\end{pmatrix},\begin{pmatrix} 5 & 8 & 6 \\  8 & 6 & 5\end{pmatrix}\right).
\end{align*}
\end{proof}

\section{Exceptional cases of Theorems~\ref{Thm:Main:Shelling} and~\ref{Thm:Main:Geometric}}\label{Section:Exceptionals}  
The exceptional cases of Theorems~\ref{Thm:Main:Shelling} and~\ref{Thm:Main:Geometric} are treated here, 
in this short, and largely independent, section.  
Both Theorem~\ref{Thm:Main:Geometric} and Theorem~\ref{Thm:Main:Shelling} 
are trivial when $a(d)$ is $0$, $1$, or $n$, the rank of $W$.  The remaining 
exceptional cases are listed in Table~\ref{Table:Exceptionals},
which reveals the fact that in the majority of these cases one has both $a(d)=2$ and $d$ a \emph{regular number}, defined below.  
After establishing the main theorems in this case 
(Corollaries~\ref{Geometric:Cor:2} and~\ref{Shellable:Cor:2} below), 
we sharpen Corollary~\ref{Cor:Independent:Maximals}:

\begin{theorem}\label{Thm:Strong:Independence}
Let $W$ be a reflection group, and let $\zeta$ and $\zeta'$ be roots of 
unity of 
orders $d$ and $d'$, respectively.  Then $E(W,\zeta)=E(W,\zeta')$ if and only if $A(d)=A(d')$.
\end{theorem}

\noindent Lastly, we discuss the reductions that one obtains in the remaining cases 
by employing Theorem~\ref{Thm:Strong:Independence} and Proposition~\ref{Prop:Maximal}\eqref{Prop:General:2}, 
and the straightforward computer verifications that result.

\begin{table}[htb]
{
\begin{tabular}{ccl@{\hskip 1.5cm}ccl@{\hskip 1.5cm}ccl }
\toprule
$W$ & $a(d)$ & $d$  &$\phantom{W}$ & $\phantom{a(d)}$ & $\phantom{d}$  &$\phantom{W}$ & $\phantom{a(d)}$ & $\phantom{d}$ \\
\midrule
$G_{25}$&      2                              &            2,6                          &  
$G_{32}$&      2                              &            4,12                         &
$G_{36}$&      2\makebox[0cm]{\, *}                             &            4                            
\\ 

$G_{28}$&      2                              &            3,6                          & 
$G_{33}$&      2\makebox[0cm]{\, *}           &            4                            & 
       &      3                              &            3,6                          
\\
       &      2                               &            4                           &       
       &      3                              &            3,6                          &
$G_{37}$&      2                              &            5,10                         
\\       

$G_{30}$&      2                              &            3,6                          &     
$G_{34}$&      2\makebox[0cm]{\, *}           &            4,12                         &
       &      2                               &            8                               
\\
       &      2                               &            4                            &  
$G_{35}$&      2                              &            4                            &
       &      2                               &            12
\\     
       &      2                               &            5,10                         &           
       &      2                               &            6                            & 
       &      4                              &            3,6  
\\

$G_{31}$&      2                              &            3,6,12                       &       
       &      3                              &            3                            &  
       &      4                               &            4                            

\\
       &      2                               &            8                            &             
       &      4                              &            2                            & 
       &                                     &
\\
\bottomrule
\end{tabular}
}
\tableskip
\caption{
All instances of an exceptional reflection group $W$ of rank $n$ and positive integer $d$ such that $a(d)\neq 0,1,n$.  
\emph{Nonregular} cases are indicated by $^*$, and values $d,d'$ 
appear together if and only if $A(d)=A(d')$; see Theorem~\ref{Thm:Strong:Independence}.
See also Table~\ref{Big:Table} below.}
\label{Table:Exceptionals}
\end{table}

We start by recalling some facts from Springer's theory of regular elements~\cite{Springer}, in which   
an element $g$ of a (finite) reflection group $W\subset\GL(V)$ is called \emph{$\zeta$-regular} if it has 
a $\zeta$-eigenvector $v$ that is not contained in any reflecting hyperplane for $W$.  
When such an element $g$ exists, the eigenvalue $\zeta$ is called a \emph{regular eigenvalue}, 
and the order $d$ of $\zeta$ is a \emph{regular number}.  For such a number $d$, 
Springer~\cite{Springer} showed that  
an element $h\in W$ is $\zeta$-regular if and only if $\dim V(h,\zeta)=a(d)$, 
and by a result of Springer and Lehrer~\cite{Lehrer:Springer:Canada}, 
these regular numbers $d$ of $W$ are easy to compute:

\begin{theorem}[Lehrer-Springer]\label{Deg:Codeg:Thm}
For any complex reflection group, a positive integer $d$ is a regular number if and only if 
it divides as many degrees as it does codegrees.
\end{theorem}

\begin{lemma}\label{Regular:Lemma}
Let $W$ be a reflection group and let $\zeta$ be a primitive $d^{\text{th}}$ root of unity.  Suppose that 
$d$ is regular and that $E=V(g,\zeta)$ is a maximal $\zeta$-eigenspace of $W$ under inclusion.  
Then for any reflection $r\in W$ one has $E\cap H_r=V(gr,\zeta)$, where $H_r:=\ker(1-r)$ denotes the reflecting 
hyperplane of $r$.  
\end{lemma}

\begin{proof}
Suppose that the trivial inclusion $E\cap H_r\subseteq V(gr,\zeta)$ is proper.  Then 
by considering dimension, $V(gr,\zeta)$ is necessarily maximal.  
It follows from a standard argument using 
Proposition~\ref{Prop:Springer:Transitive}\eqref{Springer:2} 
and a theorem of Steinberg~\cite[Thm. 1.5]{Steinberg} that
$gr$ is therefore conjugate to $g$.  
In particular, $\det(gr)=\det(g)$.  But $\det(r)\neq 1$.
\end{proof}

Applying Proposition~\ref{Prop:Maximal}\eqref{Prop:General:2} to 
the case when $\dim E=2$ gives the following.

\begin{corollary}\label{Geometric:Cor:2}
Maintain the notation and assumptions of Lemma~\ref{Regular:Lemma}, and suppose in addition 
that $a(d)=2$.  Then $E(W,\zeta)$ 
satisfies~\eqref{Main:Geometric:Equation}.
\end{corollary}

Another consequence of Lemma~\ref{Regular:Lemma} is that $E(W,\zeta)$ is connected.

\begin{corollary}\label{Connected:Corollary}
Maintain the notation and assumptions 
of Lemma~\ref{Regular:Lemma}, and suppose in addition that $a(d)\geq 2$.  Then (the Hasse diagram of) the 
poset $E(W,\zeta)\diff\{\hat{1}\}$ is connected (as a graph).  Equivalently, 
 $\Delta(E(W,\zeta)\diff\{\hat{1}\})$ is connected.
\end{corollary}

\begin{proof}
Consider two maximal eigenspaces $E,E'\in E(W,\zeta)$ and 
choose an element $g\in W$ such that $E'=gE$ (possible by Proposition~\ref{Prop:Springer:Transitive}\eqref{Springer:2}).  
Write $g$ as a product of reflections $r_kr_{k-1}\cdots r_1$ and define $E_i=r_ir_{i-1}\cdots r_1E$ for each $i$ so that
\[E=E_0,\ E_1,\ E_2,\ \ldots,\ E_k=E'\]
is a sequence of maximal eigenspaces.
Since the intersection of a neighboring pair  $E_i,E_{i+1}$ is $H_{r_{i+1}}\cap E_i$, an element 
of $E(W,\zeta)$ by Lemma~\ref{Regular:Lemma}, we conclude that $E$ is connected to $E'$ in 
the Hasse diagram of $E(W,\zeta)$, and hence the result.
\end{proof}

\begin{corollary}\label{Shellable:Cor:2}
Maintain the notation and assumptions of Lemma~\ref{Regular:Lemma}, and suppose in addition 
that $a(d)=2$.  Then  $\widehat{E(W,\zeta)}$ is CL-shellable.
\end{corollary}

\noindent{\bf{Proof of Theorem~\ref{Thm:Strong:Independence}.}}
By Proposition~\ref{Prop:Decomposition}, we may assume that $W$ is irreducible.  

Assume that $A(d)\neq A(d')$.  It is well-known~\cite[Proof of Thm. 3.4(i)]{Springer} that 
for any $S\subset[n]$, each irreducible component of $\bigcap_{i\in S} H_i$ has dimension $n-|S|$, 
where recall from Proposition~\ref{Prop:Springer:Intersection} that 
$H_i$ is the hypersurface defined by the invariant polynomial $f_i$ of a 
set of basic invariants $f_1,f_2,\ldots, f_n$ whose respective degrees are $d_1,d_2,\ldots, d_n$.  
In particular, 
\begin{equation}
\textstyle{\bigcap_{\, d\, \nmid\, d_i}} H_i\neq 
( \textstyle{\bigcap_{\, d\, \nmid\, d_i}} H_i )\cap
( \textstyle{\bigcap_{\, d'\, \nmid\, d_i}} H_i ), \label{Intersection:Equation}
\end{equation}
for indeed, $n-|A(d)|$ is not equal to $n-|A(d)\cup A(d')|$.  It follows from~\eqref{Intersection:Equation} that 
$\textstyle{\bigcap_{d\, \nmid\, d_i}} H_i \neq \textstyle{\bigcap_{d'\, \nmid\, d_i}} H_i$,
and so $E(W,\zeta)\neq E(W,\zeta')$ by Proposition~\ref{Prop:Springer:Intersection}.

Assume that $A(d)=A(d')$ so that $a(d)=a(d')$, and consider the following cases.

\noindent{\sf Case 1.} Either $W=G(m,p,n)$ or $a(d)\in\{0,1,n\}$.  
Theorem~\ref{Thm:Main:Geometric} is certainly true when $a(d)\in\{0,1,n\}$, 
and was established in \S\ref{Section:Geometric} for $G(m,p,n)$.  We thus have
\begin{align}
E(W,\zeta)&=\{E\cap X\ :\ X\in \L_W\ \text{and}\ E\in E(W,\zeta)\ \text{maximal}\},\label{eq:eig:1}\\
E(W,\zeta')&=\{E\cap X\ :\ X\in \L_W\ \text{and}\ E\in E(W,\zeta')\ \text{maximal}\}.\label{eq:eig:2}
\end{align}
Now observe that the right sides of~\eqref{eq:eig:1} and~\eqref{eq:eig:2} agree by Corollary~\ref{Cor:Independent:Maximals}.

\noindent{\sf Case 2.} $W$ exceptional, $a(d)=2$, and $d$ regular.
In this case~\eqref{eq:eig:1} and~\eqref{eq:eig:2} follow from Corollary~\ref{Geometric:Cor:2}, and 
so again $E(W,\zeta)=E(W,\zeta')$ by Corollary~\ref{Cor:Independent:Maximals}.

\noindent{\sf Case 3.} $W$ exceptional, and either $3\leq a(d)\leq n-1$ or $d$ is not regular.  
It clearly suffices to assume that $d$ is the smallest (positive) value for 
which $A(d)=A(d')$,  
so that $\zeta$ and $\zeta^{-1}$ are the only 
primitive $d^{\text{th}}$ roots of unity by Table~\ref{Table:Exceptionals}.  If $d=d'$, then 
the equality $E(W,\zeta)=E(W,\zeta')$ follows from the fact that  
$V(g,\zeta)=V(g^{-1},\zeta^{-1})$.
If $d< d'$, the classification 
shows that there exists a prime $p$ such that $\bmu_{d'}=\bmu_p\bmu_d$ and $a(p)=n$.
Choose $\zeta_p$ of order $p$ and $\zeta_d$ of order $d$ such that 
$\zeta'=\zeta_p\zeta_d$ and note that $\zeta_p\in W$ by 
Proposition~\ref{Cor:Springer}\eqref{Cor:Springer:3}.
Since $V(g,\zeta_p\zeta_d)=V(\zeta_p^{-1}g,\zeta_d)$, it follows that 
$E(W,\zeta')=E(W,\zeta_d)$, while $E(W,\zeta_d)=E(W,\zeta)$ from the previous case $d=d'$.
\qed
\pagebreak
\begin{table}[htb]
{
\begin{tabular}{l@{\hskip 1.5cm}l}
\toprule
$W$ & $\zeta$\\
\midrule
$G_{35}$ & $-1$\\
$G_{33},G_{35},G_{36},G_{37}$ & $e^{2\pi i /3}$\\
$G_{33},G_{34},G_{36},G_{37}$ & $e^{2\pi i /4}$\\
\bottomrule
\end{tabular}
}
\tableskip
\caption{}
\label{Table:Comp}
\end{table}
The remaining cases of 
Question~\ref{Question:Main:Shelling} and 
of Theorem~\ref{Thm:Main:Geometric} are those listed in Table~\ref{Table:Comp}.
We used the computer algebra software {\sc Magma} to verify Theorem~\ref{Thm:Main:Geometric} 
in each of these cases by first 
choosing\footnote{For example, $g$ may be taken to be any 
\emph{Coxeter element} raised to the power $d_n/d$ in each of the regular cases except 
 $W=G_{37},\zeta=e^{2\pi i/4}$ when $d\nmid d_n$.}
an element $g\in W$ whose eigenspace $E:=V(g,\zeta)$ is maximal 
(i.e., of dimension $a(d)$ by Proposition~\ref{Prop:Springer:Transitive}\eqref{Springer:1}), then 
constructing and checking that the set $\{E\cap X\ :\ X\in \L_W\}$ is contained in the set 
\[\{V(g r_1\cdots r_k,\zeta)\ :\ r_i\ \text{is a reflection of $W$ and }0\leq k\leq a(d)\},\]
from which the desired equality~\eqref{Main:Geometric:Equation} follows by Proposition~\ref{Prop:Maximal}\eqref{Prop:General:2}.

For the remaining cases of Theorem~\ref{Thm:Main:Shelling}, we 
first constructed the upper interval $[E,\hat{1}]$ 
(with $g$ and $E$ as above), then from the transitivity of  
$W$ on its maximal $\zeta$-eigenspaces, 
we constructed all such maximal intervals of $E(W,\zeta)$ 
by taking the $W/N_W(E)$-orbit of $[E,\hat{1}]$.  When 
$g$ is regular, we used the additional fact that 
the normalizer $N_W(E)$ is the centralizer $Z_W(g):=\{h\in W\ :\ hg=gh\}$ of $g$ in $W$.
From the collection of these upper intervals, in 
the three nonregular cases ($G_{33},G_{34},G_{36}$ with $\zeta=e^{2\pi i/4}$ so that $a(d)=2$)
we easily verified that 
$E(W,\zeta)$ is connected, so that $\widehat{E(W,\zeta)}$ is CL-shellable.  
For the single remaining case $W=G_{33}$, $\zeta=e^{2\pi i/3}$, 
we chose at random one hundred orderings $E=E_0,E_1,\ldots, E_{40}$ of 
the maximal eigenspaces (atoms) of $E(W,\zeta)$ with the property that 
\[{\rm{dist}}(E,E_1)\leq {\rm{dist}}(E,E_2)\leq{\rm{dist}}(E,E_3)\leq \cdots\leq {\rm{dist}}(E,E_{40}),\]
where ${\rm{dist}}(E,E_i)$ denotes the graph-theoretic distance from $E$ to $E_i$ in 
the restriction of $E(W,\zeta)$ to its bottom two ranks, i.e., to those eigenspaces of 
dimension either $a(d)$ or $a(d)-1$.  We found that each of these orderings was, in fact, 
a recursive atom ordering.  

\begin{table}[htb]
\centering
{
{\hskip 0mm}
\begin{tabular}{l@{ }l@{\hskip 1.5cm}l@{\hskip 0cm}l@{\hskip .5cm}l}
\toprule
$W$    &              &&$\zeta$&  \\
\midrule
$G_{35}$& (type $E_6$) &&$-1$& $e^{2\pi i/3}$\\
$G_{36}$& (type $E_7$) &&$e^{2\pi i /3}$& \\
$G_{37}$& (type $E_8$) &&$e^{2\pi i /3}$& $e^{2\pi i /4}$\\
\bottomrule
\end{tabular}
}
\tableskip
\caption{The open cases of Question~\ref{Question:Main:Shelling}
}
\label{OpenCases}
\end{table}

The remaining cases of Question~\ref{Question:Main:Shelling} 
are listed in Table~\ref{OpenCases}.  As evidence, we have in fact found 
(an unenlightening) recursive atom ordering for $E(G_{35},-1)$, and also used {\sc Magma} to show that 
$E(W,\zeta)$ is Cohen-Macaulay when $W=G_{35},\zeta=e^{2\pi i/3}$.  
We are hopeful that 
such computations can be carried out for the remaining groups.

\section{Consequences for the topology of eigenspace arrangements}\label{Section:Consequences}
Recall that the order complex of a finite poset $P$ 
is the (abstract) simplicial complex $\Delta(P)$ consisting 
of all totally ordered sets $x_1<x_2<\cdots<x_i$ in $P$.  
We adopt the convention 
of writing $\widetilde{H}_i(P)$ for its $i$th reduced homology group $\widetilde{H}_i(\Delta(P),\mathbb C)$, and 
$\widetilde{H}_i(x,y)$ for $\widetilde{H}_i(\Delta((x,y)),\mathbb C)$, where $(x,y)$ denotes 
the open interval formed by $x,y\in P$.  
If $P$ is a $G$-poset, then $\Delta(P)$ inherits an action of $G$ from $P$, and thus each
$\widetilde{H}_i(P)$ may be regarded as a $G$-module (i.e., a $\mathbb C[G]$-module) 
by functoriality.  One may alternatively 
consider the module afforded by the $i$th reduced cohomology group 
$\widetilde{H}^i(P)=\widetilde{H}^i(\Delta(P),\mathbb C)$, but 
this representation is simply dual to $\widetilde{H}_i(P)$.  

\begin{remark}  
We adopt the following conventions for the rank and 
order complex of an empty open interval $I=(x,y)$ 
and the empty poset $P\diff\{\hat{1}\}$ when $|P|=1$. 
If $y$ covers $x$, we set $r(I)=r(P\diff\{\hat{1}\})=-1$ and 
take $\Delta(x,y)$ and $\Delta(P\diff\{\hat{1}\})$ 
to be the $(-1)$-dimensional complex $\{\varnothing\}$ 
containing only the empty 
face; if $x=y$, we set $r(I)=-2$ and regard $\Delta(x,x)$ 
as a $(-2)$-dimensional 
\emph{degenerate empty complex} $\varnothing$ with 
no faces at all.  Then 
$\widetilde{H}_{i}(\{\varnothing\})$ is $\mathbb C$ 
when $i=-1$, and 
$0$ otherwise, while we let the reduced homology of 
$\varnothing$ 
vanish in all dimensions 
and define $\widetilde{H}_{-2}(\varnothing):=\mathbb C$.   
\end{remark}

For our purposes, call a finite collection $\mathcal A$ of proper complex linear subspaces of $\mathbb C^n$ an 
\emph{arrangement}.  Associate with $\mathcal A$ its {\it intersection lattice} $L(\mathcal A)$, 
obtained by ordering all intersections of subspaces in $\mathcal A$ by reverse inclusion.  Note that  
the top element $\hat{1}$ of $L(\mathcal A)$ is the (nonempty) intersection over $\mathcal A$ 
and that the bottom element $\hat{0}$ is $\mathbb C^n$, the intersection over the empty set.  
The following celebrated 
result of Goresky and MacPherson~\cite{GoreskyMacPherson} says that the cohomology of the complement 
$\mathcal M_{\mathcal A}:=\mathbb C^n\diff \bigcup_{X\in \mathcal{A}}X$ is completely determined by the combinatorial data of the 
arrangement.  

\begin{theorem}[Goresky-MacPherson]\label{GS}  Let $\mathcal A$ be a complex arrangement.  Then
\[\widetilde{H}^i(\mathcal M_{\mathcal A})=\bigoplus_{x\in L(\mathcal A)\diff\{\hat{0}\}}
\widetilde{H}_{2\, \codim\, x-i-2}(\hat{0},x)\quad\text{for all $i$.}\]
\end{theorem}

Call an arrangement $\mathcal A$ a \emph{$G$-arrangement} if $G\subset \GL(\mathbb C^n)$ is such that $gX\in \mathcal A$ for every 
$X\in\mathcal A$ and $g\in G$.  For such an arrangement both the cohomology of $\mathcal M_{\mathcal A}$
and the homology of $L(\mathcal A)$ inherit a $G$-module structure, and 
Sundaram and Welker~\cite{Sundaram:Welker} established the following equivariant formulation of Theorem~\ref{GS}.

\begin{theorem}[Sundaram-Welker]\label{Thm:Sundaram:Welker}
For a (finite, complex, linear) $G$-arrangement $\mathcal A$, one has a $G$-module isomorphism
\begin{equation}\widetilde{H}^i(\mathcal M_{\mathcal A})\cong\bigoplus_{x}
\Ind_{G_x}^G\widetilde{H}_{2\, \codim\, x-i-2}(\hat{0},x)\qquad\text{for all $i$},\label{Eq:Sundaram:Welker}\end{equation}
where $x$ runs over a collection of representatives for the $G$-orbits on $L(\mathcal A)\diff\{\hat{0}\}$, and $G_x$ denotes 
the stabilizer $\{g\in G\ :\ gx=x\}$ for $x\in L(\mathcal A)$. 
\end{theorem}

A consequence of Theorem~\ref{Thm:Main:Shelling}, Theorem~\ref{Cor:Dowling}, and Theorem~\ref{Thm:Sundaram:Welker} is that 
for $W=G(m,p,n)$, 
the top nonvanishing cohomology module $\widetilde{H}^{*}(\mathcal M(W,d))$ of the complement $\mathcal M(W,d)$ 
of the proper eigenspaces of $W$ (see \S\ref{Section:Introduction}) 
is equivariantly isomorphic to the top homology module $\widetilde{H}_*(\overline{E(W,\zeta)})$ of the proper 
part of the corresponding poset of eigenspaces:

\begin{proposition}\label{Cor:Cohomology}  Let $W$ be an irreducible non-exceptional reflection group of rank $n$.  That is, 
$W$ is either $W(A_n)$ or $G(m,p,n)$ with $m>1$ and 
$(m,p,n)\neq (2,2,2)$.  Let $\zeta$ be a primitive $d^{\text{th}}$ root of unity and let 
$\mathcal A$ be the arrangement of all proper $\zeta$-eigenspaces of $W$ so 
that $\mathcal M(W,d)=\mathcal M_{\mathcal A}$.   
Then $L(\mathcal A)=E(W,\zeta)\cup \{\mathbb C^n\}$ and the following hold:
\begin{enumerate}[(i)]
\item  If $a(d)=n$, so that $\mathbb C^n\in E(W,\zeta)$, then as $W$-modules
\[\widetilde{H}^{2n-a(d)}(\mathcal M(W,d))
\ \cong\ \widetilde{H}_{a(d)-2}(\overline{E(W,\zeta)}).\] \label{Complement:Corollary:1}
\item  If $a(d)\neq n$, so that $\mathbb C^n\not\in E(W,\zeta)$, then as $W$-modules
\[\widetilde{H}^{2n-a(d)-1}(\mathcal M(W,d))
\ \cong\ \widetilde{H}_{a(d)-1}(\overline{E(W,\zeta)}).\]\label{Complement:Corollary:2}
\end{enumerate}
\end{proposition}

\begin{proof}  
The first claim, that $L(\mathcal A)=E(W,\zeta)\cup\{\mathbb C^n\}$, follows from the fact that 
$E(W,\zeta)$ is an upper order ideal of an intersection lattice $\L_{W'}$; see Theorem~\ref{Cor:Dowling}.  
Supposing that $a(d)\neq n$, the intersection lattice $L(\mathcal A)$ is obtained from $E(W,\zeta)$ by adjoining the new element 
$\mathbb C^n$.  
Because the top element of $L(\mathcal A)$ has dimension $0$ by irreducibility, it follows that the rank of an interval $(\hat{0},x)$ 
is equal to $a(d)-\dim x-1$ whenever $x\neq\hat{0}$. 
Taking $i$ to be $2n-a(d)-1$, the homology summands 
$\widetilde{H}_{2\codim x-i-2}(\hat{0},x)$ in~\eqref{Eq:Sundaram:Welker} become $\widetilde{H}_{a(d)-2\dim x-1}(\hat{0},x)$.
Since shellability is inherited by open intervals, such a summand vanishes if $a(d)-2\dim x-1$ is not equal to 
the dimension $a(d)-\dim x-1$ of $\Delta(\hat{0},x)$, and so the only surviving summand of~\eqref{Eq:Sundaram:Welker}  
corresponds to the top element $x=\hat{1}$ of dimension $0$.  The second claim follows, and the first is similar.
\end{proof}

The first case of Corollary~\ref{Cor:Cohomology} is well-known, since $E(W,1)=\L_W$ is a geometric lattice.  
The homology module of $\Delta(\overline{\L_W})$ for $W=W(A_n)$ has received considerable attention; 
see Section~\ref{Section:Introduction}.  
In this case Stanley~\cite{Stanley:Aspects} used work of Hanlon~\cite{Hanlon} to express the top 
homology module, and 
thus the top cohomology module of $\mathcal M(W(A_n),1)$, as an induced linear representation.  
Lehrer and Solomon~\cite{Lehrer:Solomon} extended Stanley's results to all cohomology groups of 
$\mathcal M(W(A_n),1)$, and conjectured an extension to all (complexified) finite Coxeter groups.  In 
the next section we will extend Stanley's result to arbitrary eigenvalues $\zeta$.  
We have not explored the possibility of extending Lehrer and Solomon's result.

\section{Combinatorics of $E(W,\zeta)$ in type A}\label{Section:Type:A}
Recall that $E(\mathfrak S_n,1)$ is the well-understood intersection lattice for the braid 
arrangement.  As such, it may be considered as the lattice of all set partitions of $\{1,2,\ldots, n\}$ ordered by refinement.  
This simple combinatorial model extends to the intersection lattice for $G(r,1,n)$ through a construction of 
Dowling~\cite{Dowling}, which in turn provides a model for $E(W,\zeta)$ via Theorem~\ref{Cor:Dowling} when 
$W=G(m,p,n)$ and $E(W,\zeta)\neq \L_W$.  
Theorem~\ref{Cor:Dowling} also tells us that 
the model inherited by $E(W,\zeta)$ is again particularly simple when the minimal elements have a 
simple description as elements of $\L_{W'}$, where $W'=G(m\vee d,1,n)$ and $d$ is the order of $\zeta$.  
In this section we set $W=\mathfrak S_n$ and 
$d>1$.  Not only is the resulting combinatorial description for $E(\mathfrak S_n,\zeta)$ interesting in its own right, but it 
naturally distinguishes the poset of \emph{$d$-divisible partitions} and leads to a precise description of the 
homotopy type of $E(\mathfrak S_n,\zeta)\diff\{\hat{1}\}$ in terms of integer partitions.  Throughout, $d>1$ and $\zeta$ denotes a primitive $d^{\text{th}}$ root of unity.

\subsection{Balanced Partitions}\label{Section:Balanced}
The purpose of this subsection is to give a combinatorial description of $E(\mathfrak S_n,\zeta)$.  As outlined above, we do so 
by viewing $E(\mathfrak S_n,\zeta)$ as an upper order ideal of $\L_W$ for $W=G(d,1,n)$, regarded as a Dowling lattice.  
Though Dowling's original notation~\cite{Dowling} is convenient for general groups, the lattice $\L_W$ is constructed 
as a Dowling lattice from the cyclic group $\mathbb Z/d\mathbb Z$ and lends itself to description in terms of integrally 
weighted set partitions, a notation that we shall use throughout.

Recall that a \emph{set partition} of $\{0,1,\ldots,n\}$ is a collection $\pi=\{B_0,B_1,\ldots, B_\ell\}$ of 
nonempty disjoint sets with $\bigcup_{i=0}^\ell B_i=\{0,1,\ldots,n\}$.  We call $\ell$ the \emph{length} of $\pi$ and write 
$\ell(\pi):=\ell$.  Henceforth, we shall assume that the \emph{blocks} 
are indexed so that $B_0$ is the \emph{zero block}, i.e., that $0\in B_0$.
We also adopt the convention of writing such a partition as $B_0/B_1/\cdots/B_\ell$, omitting commas and set braces for 
individual blocks when working with explicit sets of integers.

From a partition $B_0/B_1/\cdots/B_\ell$ of $\{0,1,\ldots,n\}$ and a positive integer $m$, one obtains an 
\emph{$m$-weighted partition} by assigning a weight $w_i\in \mathbb Z/m\mathbb Z$ to each element $i\in B_1\cup B_2\cup\cdots\cup B_\ell$, 
while elements of $B_0$ remain unweighted.  
Indicating weights with superscripts, we have 
$0 1 5 / 2^0 3^4/ 4^0$ and $023/1^2 5^2/ 4^4$ 
as examples of weighted partitions for $n=5$, with $0 2 3$ being the zero block of the latter.  We shall 
also find it useful to depict each weight by an equal number of bars, in which case we superfluously underline the zero block in order to emphasize that its 
elements bear no weights; 
see Figure~\ref{Fig:Dowling} and~\ref{Fig:EP42}. 

For $B$  a nonzero block of a weighted partition, let $B^w$ denote the block obtained from $B$ by adding 
$w$ to the weight of each of its elements (modulo $d$), 
and let $B^\z$ denote the block obtained by removing all weights.  
For example, if $B=1^3 3^0 4^4$ and $d=5$, then $B^3= 1^1 3^3 4^2$ and $B^\z=1 3 4$.

Since $d>1$, the arrangement of reflecting hyperplanes of $G(d,1,n)$ is given by
\[\mathcal A_{n,d}:=\{z_i=0\ :\ 1\leq i \leq n\}\cup\{z_i=\zeta^w z_j\ :\ 1\leq i<j\leq n,\ w\in\mathbb Z/d\mathbb Z\}.\]
The intersection lattice $L(\mathcal A_{n,d})$ is $\mathfrak S_n$-equivariantly isomorphic 
(in fact, $G(d,1,n)$-isomorphic) to the (particular) Dowling lattice $\Pi_{n,d}$ defined as follows:
\begin{definition}
$\Pi_{n,d}$ is the poset of all ($d$-)weighted partitions 
$\pi=B_0/\cdots/B_\ell$ of $\{0,\ldots, n\}$ with partial order 
defined by setting $\pi_1\leq\pi_2$ if for each block $B$ of $\pi_1$ 
there exists a block of $\pi_2$ that either contains $B^\z$ or contains $B^w$ for some $w$; see Figure~\ref{Fig:Dowling}.  
It is the \emph{Dowling lattice} constructed from the cyclic group $\mathbb Z/d\mathbb Z$; see~\cite{Dowling}.
\end{definition}
We consider $\Pi_{n,d}$ as an $\mathfrak S_n$-poset by defining
$\sigma(B_0/\cdots /B_\ell) :=\sigma B_0/\cdots/ \sigma B_\ell$, and  
the abovementioned equivariant isomorphism $\Pi_{n,d}\stackrel{\sim}{\longrightarrow} L(\mathcal A_{n,d})$ is given by
\begin{equation}
B_0/\cdots/ B_\ell\mapsto V(B_0)\oplus\cdots\oplus V(B_\ell),\label{Equation:Isomorphism}
\end{equation} 
where $V(B_0):={\mathbf 0}$, and for any nonzero block $B=\{b_1^{w_1},b_2^{w_2},\ldots, b_k^{w_k}\}$ 
we set
\[V(B):=\{\mathbf{z}\in\mathbb C^n\ :\ \zeta^{w_1}z_{b_1}=\zeta^{w_2}z_{b_2}=\cdots=\zeta^{w_k}z_{b_k}\ \ \text{and}\ \ z_i=0\text{ for }i\not\in B\}.\]

\begin{figure}[hbt]
\center
\includegraphics[scale=.8]{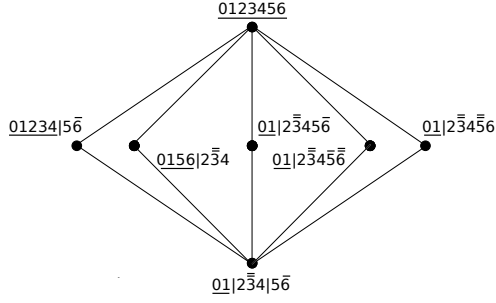}
\caption{$[x,\hat{1}]$ in $\Pi_{6,3}$ for $x=01|2^03^24^0|5^06^1=\underline{01}|2\overline{\overline{3}}4|5\overline{6}$ 
corresponding to    
$\{\mathbf{z}\in\mathbb C^n\ :\ z_1=0,\ z_2=\zeta^2z_3=z_4,\ z_5=\zeta z_6\}$.}
\label{Fig:Dowling}
\end{figure}

\begin{definition}\label{Definition:Balanced}
Let $d>1$.  Call a partition $\pi$ of $\Pi_{n,d}$ \emph{balanced} if for each block $B$ of $\pi$, all weights 
$w\in\mathbb Z/d\mathbb Z$ appear in $B$ with equal multiplicity.  The $\mathfrak S_n$-subposet 
of $\Pi_{n,d}$ of all balanced partitions is denoted $\Pi_n^d$.
\end{definition}

\begin{figure}[hbt]
\center
\includegraphics[scale=.8]{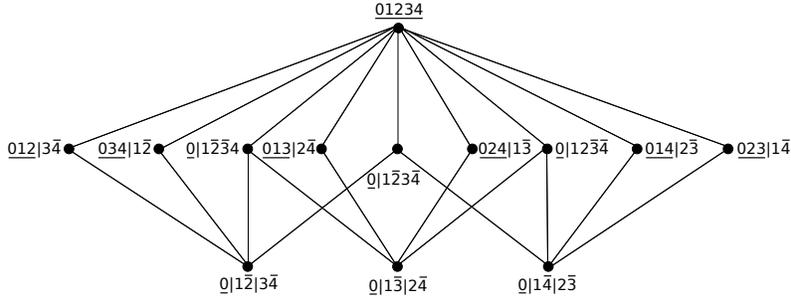}
\caption{$\Pi_4^2$ with each weight replaced by an equal number of bars.}\label{Fig:EP42}
\end{figure}

Recall that an \emph{upper order ideal} of a poset $P$ is a subposet $I$ with the property that $y\in I$ whenever 
$x\leq y$ for some $x\in I$.  The minimal elements of $I$ are called \emph{generators}.
With this terminology, the next observation follows directly from the definitions of $\Pi_{n,d}$ and $\Pi_n^d$.

\begin{proposition}\label{Prop:Filter}
$\Pi_n^d$ is the upper order ideal of $\Pi_{n,d}$ that is generated by the $\mathfrak S_n$-orbit of the element
$\pi=B_0/B_1/\cdots/B_{a(d)}$, where $a(d)=\left\lfloor\frac{n}{d}\right\rfloor$ and 
\begin{align*}
B_0& = \{0\}\cup\{\ a(d)\cdot d+1\ ,\ a(d)\cdot d+2\ ,\ \ldots\ ,\ n\ \}\\
\intertext{while}
B_1& =\{1^0,2^1,\ldots, d^{d-1}\}\\
B_2& =\{(d+1)^0,(d+2)^1,\ldots, (2d)^{d-1}\}\\
&{\hskip .15cm}\vdots\\
B_{a(d)}& =\{((a(d)-1)\cdot d+1)^0,\ldots, (a(d)\cdot d)^{d-1}\}.
\end{align*}
\end{proposition}

\begin{remark}\label{Remark:Rank}
Note that $r(\Pi_n^d)=a(d)$ and $\ell(x)=r([x,\hat{1}])$ for $x\in \Pi_n^d$.  Also observe 
that $\Pi_n^d$ has a bottom element $\hat{0}$ if and only if $n=d=2$ or $d>n$.  In the latter case, 
the top and bottom elements coincide, i.e., $\Pi_n^d$ is the trivial poset consisting of the element $\underline{01\cdots n}$.  
\end{remark}

\begin{theorem}\label{Model}
Let $d>1$.  Then $E(\mathfrak S_n,\zeta)\cong \Pi_n^d$ as $\mathfrak S_n$-posets.
\end{theorem}

\begin{proof}
Because $\Pi_n^d$ and $E(\mathfrak S_n,\zeta)$ are upper order ideals of $\Pi_{n,d}$ and $L(\mathcal A_{n,d})$, respectively, 
it suffices to show that the minimal elements of $\Pi_n^d$ correspond to those of $E(\mathfrak S_n,\zeta)$ under the isomorphism 
$\Pi_{n,d}\rightarrow L(\mathcal A_{n,d})$ given in~\eqref{Equation:Isomorphism}.  In fact, since the 
isomorphism is equivariant, by transitivity it suffices to see that the distinguished minimal element of 
$\Pi_n^d$ in Proposition~\ref{Prop:Filter} corresponds to some maximal $\zeta$-eigenspace of $\mathfrak S_n$.  But this is clear, 
for the distinguished element corresponds to the $\zeta$-eigenspace of 
\[(1,2,\ldots, d)\ (d+1,d+2,\ldots, 2d)\ \cdots\ (\ (a(d)-1) d+1\ ,\ (a(d)-1) d+2\ ,\ \ldots\ ,\ a(d)d\ ),\]
which is maximal under inclusion by considering dimension.  
\end{proof}

A general feature of Dowling lattices is that their upper intervals $[x,\hat{1}]$ are again Dowling 
lattices; see~\cite[Thm. 2]{Dowling}.  
For $x\in \Pi_{n,d}$ one can see this by writing $x=B_0/B_1/\cdots/B_\ell$ and then constructing a poset 
isomorphism $[x,\hat{1}]\rightarrow \Pi_{\ell(x),d}$ by mapping any block of the form 
$B_{i_1}^{w_1}\cup\cdots\cup B_{i_j}^{w_j}$ with $w_i\in\mathbb Z/d\mathbb Z$ to $\{i_1^{w_1},\ldots,i_j^{w_j}\}$ and mapping 
any zero block of the form $B_0\cup B_{i_1}^\z\cup\cdots\cup B_{i_j}^\z$ to zero block $\{0,i_1,\ldots, i_j\}$.  

\begin{lemma}\label{Upper:Intervals}
$[x,\hat{1}]\cong \Pi_{\ell(x),d}$ for each $x\in\Pi_n^d$.  
\end{lemma}

\subsection{The topology of balanced partitions}\label{Section:Topology}  
Recall that $E(\mathfrak S_n,\zeta)\diff\{\hat{1}\}$ being shellable implies that its order complex is homotopy 
equivalent to a bouquet of spheres.  The aim of this subsection is to calculate the exact number of such spheres.  
Central to our analysis is the following consequence of the Hopf trace formula.

\begin{theorem}[Sundaram~\cite{Sundaram}]\label{Theorem:Sundaram}
For $P$ a Cohen-Macaulay $G$-poset with top element $\hat{1}$, one has an isomorphism of virtual $G$-modules
\[\widetilde{H}_{r(P)-1}(P\diff\{\hat{1}\})
\cong
\bigoplus_{i=0}^{r(P)}\ (-1)^{r(P)+i}\bigoplus_{x^G\in P/G}\Ind_{G_x}^G\, \widetilde{H}_{i-2}(x,\hat{1}).
\]
\end{theorem}

We shall also need the following consequence of Dowling's well-known description of 
the M\"obius function for any Dowling lattice~\cite{Dowling} and Lemma~\ref{Upper:Intervals}.

\begin{lemma}\label{Lemma:Dowling}
Let $x\in \Pi_n^d$.  Then $\dim\widetilde{H}_{\ell(x)-2}(x,\hat{1})=\prod_{i=0}^{\ell(x)-1}(1+id)$.
\end{lemma}

Regard an \emph{integer partition} $\lambda$ of $n\geq 0$ as a 
multiset of positive integers $\lambda_1,\ldots,\lambda_\ell$ with 
$\sum \lambda_i=n$, and write $|\lambda|=n$ or $\lambda\vdash n$.  
The integer $\ell(\lambda):=\ell$ is the \emph{length} of $\lambda$.  We shall also 
regard the partition $\lambda$ as a sequence $\lambda=(\lambda_1^{m_1},\ldots,\lambda_\wp^{m_\wp})$ 
of distinct positive integers $0<\lambda_1<\lambda_2<\cdots<\lambda_\wp$ 
with positive weights $m_i$ satisfying $\sum_{i=1}^\wp m_i\lambda_i=n$.  The integer $\wp(\lambda):=\wp$ 
is the number of distinct parts of $\lambda$.  
In what follows, the convention being used will be made clear by the presence or absence of weights $m_i$.

A \emph{pointed partition} $\Lambda=(\lambda_0,\lambda)$ of $n\geq 0$ 
is a nonnegative integer $0\leq\lambda_0\leq n$ together with a partition $\lambda$ of $n-\lambda_0$.  
Such a partition naturally arises in Dowling lattices as the \emph{type} of a partition 
$x=B_0/B_1/\cdots/B_\ell$, defined to be
\[\type(x)\Def(|B_0\diff\{0\}|,\{|B_1|,\ldots,|B_\ell|\}).\]

\begin{corollary}\label{Cor:Formula}
  $\Delta(\Pi_n^d\diff\{\hat{1}\})$ is homotopy equivalent to a bouquet of
\begin{equation}
\sum_{0\leq |\lambda| \leq a(d)}\frac{(-1)^{a(d)+\ell(\lambda)}}{d^{\ell(\lambda)}(n-d|\lambda|)!}\ 
\frac{n!}{\prod_{j=1}^p \lambda_j!^{m_jd}\,m_j!}
\prod_{i=0}^{\ell(\lambda)-1}(1+id)\label{homotopy:type:formula}
\end{equation}
many $(a(d)-1)$-spheres.
\end{corollary}

\begin{proof}
Start by noting that two partitions $x,y\in \Pi_n^d$ are in the same 
$\mathfrak S_n$-orbit if and only if $\type(x)=\type(y)$, 
and observe that for $\type(x)=(\lambda_0,\lambda)$ one has
\begin{equation}\label{Equation:Stab}
|\Stab_{\mathfrak S_n}(x)|=d^{\ell(\lambda)}\lambda_0!\prod_{i=1}^\wp m_i!(\lambda_i/d)!^{m_id}.
\end{equation}
Taking dimensions in Theorem~\ref{Theorem:Sundaram} and employing~\eqref{Equation:Stab} and Lemma~\ref{Lemma:Dowling} 
gives the result. 
\end{proof}

\noindent
We list some initial values of~\eqref{homotopy:type:formula} in Table~\ref{Table:Homology:Dimensions}.
\begin{table}[hbt]
\center
\begin{tabular}{l@{\hskip .6cm}llllllll}
\toprule
$n\backslash d$
& 2 & 3 & 4 & 5 & 6 & 7 & 8 & 9\\ 
\midrule
2 &0  &  &   &  &  &  &  &    \\
3 &2  &1  &   &  &  &  &  &    \\
4 &1  &7  &5   &  &  &  &  &    \\
5 &21  &19  &29   &23  &  &  &  &    \\
6 &24  &91  &89   &143  &119  &   &  &  \\
7 &510  &841  &209   &503  &839  &719  &  &  \\
8 &918  &3529  &5251   &1343  &3359  &5759  &5039  &   \\
9 &22246  &32367  &50275   &3023  &10079  &25919  &45359  &40319   \\ 
\bottomrule
\end{tabular}
\tableskip
\caption{$\dim\widetilde{H}_{a(d)-1}(\Pi_n^d\diff\{\hat{1}\})$ for $n\leq 9$ and $a(d)\neq 0$}\label{Table:Homology:Dimensions}
\end{table}
\begin{remark}
One may identify the stabilizer $\Stab_{\mathfrak S_n}(x)$ of $x=B_0/B_1/\cdots/B_\ell$ as a 
(suitably defined) product of nested wreath products 
\[\mathfrak S_{\lambda_0}\times 
\mathfrak S_{m_1}[C_d][\mathfrak S_{\lambda_1/d}]\times
\cdots\times\mathfrak S_{m_\wp}[C_d][\mathfrak S_{\lambda_\wp/d}],\]
from which one can show that
\[\widetilde{H}_{\ell-2}(x,\hat{1})\cong_{\Stab_{\mathfrak S_n}(x)} 
\Res^{\mathfrak S_{\lambda_0}\times\mathfrak S_{\ell}[C_d][\mathfrak S_n]}_
{\Stab_{\mathfrak S_n}(x)}\triv\otimes \widetilde{H}_{\ell-2}(\overline{\Pi}_{\ell,d})[\triv]\]
for a suitable embedding $\Stab_{\mathfrak S_n}(x)\subseteq \mathfrak S_{\lambda_0}\times\mathfrak S_{\ell}[C_d][\mathfrak S_n]$.  
Here, $ \widetilde{H}_{\ell-2}(\overline{\Pi}_{\ell,d})[\triv]$ denotes the $\mathfrak S_{\ell}[C_d][\mathfrak S_n]$-module 
$\triv^{\otimes\, \ell d }\otimes \widetilde{H}_{\ell-2}(\overline{\Pi}_{\ell,d})$ given by 
\[(\sigma_1,\ldots, \sigma_{\ell d}\ ;\, \sigma)
(v_1\otimes\cdots\otimes v_n\otimes w)= v_{1}\otimes\cdots\otimes v_{n}\otimes \sigma w\]
for $\sigma_i\in\mathfrak S_n$ and $\sigma\in\mathfrak S_{\ell}[C_d]$, and $\Pi_{\ell,d}$ is viewed as a $\mathfrak S_\ell[C_d]$-poset.  
Employing Theorem~\ref{Theorem:Sundaram}, one then obtains a (virtual) expression for $\widetilde{H}_{a(d)-1}(\Pi_n^d\diff\{\hat{1}\})$ 
as an $\mathfrak S_n$-module in terms of the $\mathfrak S_\ell[C_d]$-modules $\widetilde{H}_{\ell-2}(\overline{\Pi}_{\ell,d})$ 
studied by Hanlon~\cite{Hanlon} and Gottlieb-Wachs~\cite{GottliebWachs}.
\end{remark}

\begin{problem}
  Describe the $\mathfrak S_n$-irreducible decomposition of $\widetilde{H}_{a(d)-1}(\Pi_n^d\diff\{\hat{1}\})$ explicitly in general.
\end{problem}

Tables~\ref{Table:S4}-\ref{Table:S7} 
give the decomposition of $\widetilde{H}_{a(d)-1}(\Pi_n^d\diff\{\hat{1}\})$ 
into irreducible submodules for $n=4,5,6,7$.  
Each entry lists the multiplicity with which the Specht 
module $S^\lambda$ occurs in $\widetilde{H}_{a(d)-1}(\Pi_{|\lambda|}^d\diff\{\hat{1}\})$.

\begin{table}[hbt]
\centering
\begin{tabular}{l@{\hskip.6cm }cccc}
\toprule
$d{\hskip .05cm}\backslash\lambda$
& \includegraphics[scale=.8]{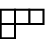}& \includegraphics[scale=.8]{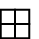}& \includegraphics[scale=.8]{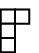}& 
{\raisebox{+0ex}[1\height][0ex]{\includegraphics[scale=.8]{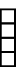}}}\\
\midrule
2&              
                                        &                                          &                                            &        1                                    \\
3&                                     
         1                              &                                          &              1                            &        1                                     \\
4&                                     
                                         &              1                          &              1                            &                                           \\
\bottomrule
\end{tabular}
\tableskip
\caption{Decomposition of the $\mathfrak S_4$-module $\widetilde{H}_{a(d)-1}(\Pi_4^d\diff\{\hat{1}\})$}
\label{Table:S4}
\end{table}
\begin{table}[hbt]
\centering
\begin{tabular}{l@{\hskip .6cm}cccccc}
\toprule
$d{\hskip .05cm}\backslash\lambda$
& \includegraphics[scale=.8]{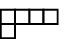}& \includegraphics[scale=.8]{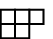}& \includegraphics[scale=.8]{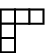}& \includegraphics[scale=.8]{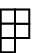}& \includegraphics[scale=.8]{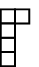}& 
{\raisebox{+0ex}[1\height][0pt]{\includegraphics[scale=.8]{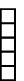}}}\\
\midrule
2&                  
                                         &              1                          &          1                                &                     1                   &                1                            &      1                                        \\   
3&                               
                 1                     &              1                          &          1                               &                                          &                1                            &                                                \\
4&                               
                 1                     &              1                          &          1                                &                     2                   &                1                            &                                                \\
5&                               
                                        &              1                          &          2                                &                     1                   &                                              &      1                                  \\
\bottomrule      
\end{tabular}
\tableskip
\caption{Decomposition of the $\mathfrak S_5$-module $\widetilde{H}_{a(d)-1}(\Pi_5^d\diff\{\hat{1}\})$}
\end{table}\begin{table}[hbt]
\centering
 \begin{tabular}{ l@{\hskip .6cm}cccccccccc }
\toprule
$d{\hskip .05cm}\backslash\lambda$
& \includegraphics[scale=.8]{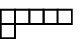}& \includegraphics[scale=.8]{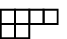}& \includegraphics[scale=.8]{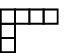}& \includegraphics[scale=.8]{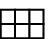}& \includegraphics[scale=.8]{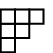}& \includegraphics[scale=.8]{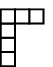}& 
 \includegraphics[scale=.8]{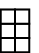}& \includegraphics[scale=.8]{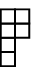}& \includegraphics[scale=.8]{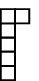}& {\raisebox{+0ex}[1\height][0ex]{\includegraphics[scale=.8]{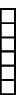}}}\\
\midrule

 2& 
                                         &                                          &                                            &                                         &                                            &              1                             & 
                                         &             1                              &                     1                         &                                                  \\
 3& 
                                         &                                          &              1                            &                                         &                  2                        &              2                             & 
           1                             &          2                                 &         1                                     &       1                                         \\
 4& 
            1                            &              2                          &               1                           &                                         &                  2                        &           1                                & 
                       1                   &         1                                  &                                                &                                                  \\
 5& 
            1                            &              1                          &         2                                 &       1                                &      4                                    &     2                                      & 
                  1                        &   1                                        &                    1                          &            1                                    \\
 6& 
                                         &                2                        &           2                               &            1                           &          2                                &  2                                         & 
                    2                      &          1                                 &         1                                     &                                              \\
\bottomrule
 \end{tabular}\tableskip\caption{Decomposition of the $\mathfrak S_6$-module $\widetilde{H}_{a(d)-1}(\Pi_6^d\diff\{\hat{1}\})$} 
 \end{table}\begin{table}[hbt]
\centering
\begin{tabular}{ l@{\hskip .6cm}cccccccccccccc }
\toprule
$d{\hskip .05cm}\backslash\lambda$
&\includegraphics[scale=.8]{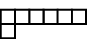}& \includegraphics[scale=.8]{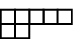}& \includegraphics[scale=.8]{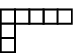}& \includegraphics[scale=.8]{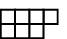}& \includegraphics[scale=.8]{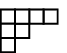}& 
\includegraphics[scale=.8]{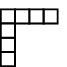}& \includegraphics[scale=.8]{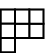}& \includegraphics[scale=.8]{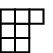}& \includegraphics[scale=.8]{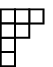}& \includegraphics[scale=.8]{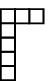}&
\includegraphics[scale=.8]{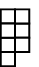}& \includegraphics[scale=.8]{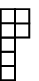}& \includegraphics[scale=.8]{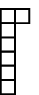}& {\raisebox{+0ex}[1\height][0ex]{\includegraphics[scale=.8]{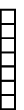}}}\\
\midrule
2&  
                                          &                                         &                                            &                       1                 &                    2                     &                 2                           &                           2                      &                  2                        &                   5                        &                 3                              &                 2                         &                3                              &      2                                        &                                                    \\
3& 
                                          &               1                        &                 2                         &                    1                    &                   5                      &          4                                  &                  3                               &                3                          &          7                                 &          4                                     &            3                              &          3                                    &     2                                         &    1                                              \\
4&  
          1                              &         2                              &       1                                   &          1                              &                2                         &       1                                     &                                                   &                 1                         &        1                                   &                                                 &                                            &                                                &                                                &                                                    \\
5& 
                1                        &          2                             &                  2                        &          1                              &          4                               &                2                            &                     3                            &          2                                &        3                                   &     1                                          &      1                                    &                                                &    1                                          &                                                    \\
6& 
                1                        &                   2                    &       2                                   &         3                               &        6                                 &            4                                &                    3                             &         4                                 &       5                                    &     3                                          &     3                                     &           2                                   &     1                                         &                                                    \\
7& 
                                         &                            2           &         3                                 &                2                        &         5                                &         2                                   &                       3                          &         3                                 &        5                                   &           3                                    &        2                                  &   2                                             &                                                &    1\\
\bottomrule
\end{tabular}\tableskip\caption{Decomposition of the $\mathfrak S_7$-module $\widetilde{H}_{a(d)-1}(\Pi_7^d\diff\{\hat{1}\})$}
\label{Table:S7}
\end{table}

\subsection{The $d$-divisible partition poset}\label{Section:Module2}
Throughout this section we shall assume that $d>1$.  
For $n+1$ divisible by $d$, recall that a set partition of $\{1,2,\ldots,n+1\}$ 
is called \emph{$d$-divisible} if all blocks have size divisible by $d$, and let $P_{n+1}^d$ denote 
the $\mathfrak S_{n+1}$-poset of all such partitions ordered by refinement.  
The \emph{$d$-divisible partition lattice} $P_{n+1}^d\cup\{\hat{0}\}$ has been extensively studied, starting 
with Sylvester~\cite{Sylvester} computing the M\"obius function for $d=2$.  Stanley~\cite{Stanley:Exponential} generalized Sylvester's 
result to arbitrary $d$ and conjectured that the restriction of the top (reduced) homology module of $P_{n+1}^d\diff\{\hat{1}\}$ to 
$\mathfrak S_n$ is the Specht module of skew ribbon shape $(d-1,d,d,\ldots, d)$; 
see Figure~\ref{Fig:Skew}.  (One instead obtains the regular representation of $\mathfrak S_n$ if $d=1$; see~\cite{Stanley:Aspects}.)    
Calderbank-Hanlon-Robinson~\cite{CHR} established Stanley's conjecture via character calculations, and Wachs~\cite{Wachs:Divisible} 
later exhibited an explicit basis for the top homology group via an EL-shelling.  
Ehrenborg and Jung~\cite{Ehrenborg} extended this phenomenon to more general posets of \emph{pointed partitions} $\Pi^\bullet_{\vec{c}}$, 
showing that each $\Pi^\bullet_{\vec{c}}\diff\{\hat{1}\}$ is either contractible or has top homology the $\mathfrak S_n$-Specht module 
of skew ribbon shape associated with the composition $\vec{c}=(c_1,c_2,\ldots,c_k)$.

\begin{figure}[hbt]
\center
\includegraphics[scale=.7,angle=180]{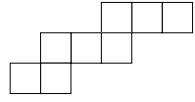}
\caption{Skew ribbon shape $(2,3,3)$}\label{Fig:Skew}
\end{figure}

In this section we present an intermediate family of posets $\Pi_n^{(d)}$, between the family of $d$-divisible partition posets 
$P_{n+1}^d$ and the family of pointed partition posets $\Pi^\bullet_{\vec{c}}$, that arises naturally in the context of eigenspace 
arrangements for $\mathfrak S_n$.  Furthermore, we conjecture that the top homology of $\Pi_n^{(d)}\diff\{\hat{1}\}$ 
is isomorphic to a submodule of the top homology of $\Pi_n^d\diff\{\hat{1}\}$.

For a pointed partition $\Lambda=(\lambda_0,\lambda)$ of $n$, let $\Pi_\Lambda\subseteq \Pi_{n,1}$ 
denote the upper order ideal generated by the elements of type $\Lambda$ and regard it as an $\mathfrak S_n$-poset 
through the usual action of $\mathfrak S_n$ on $[n]$.  
For the particular choice 
\[\Lambda=(n-d\cdot a(d)\ ,\ \{d,d,\ldots, d\}),\] write 
$\Pi_n^{(d)}$ for $\Pi_\Lambda$ and call it the \emph{pointed $d$-divisible partition poset}.  It is 
$\mathfrak S_n$-isomorphic to Ehrenborg and Jung's poset $\Pi^{\bullet}_{(d,d,\ldots, d,n-d\cdot a(d))}$ 
via the map that removes all weights and then replaces each zero block 
$B_0$ with the \emph{distinguished block} $Z:=B_0\diff \{0\}$; see~\cite{Ehrenborg} for details.  
For $n+1$ divisible by $d$, the map
\begin{equation}\label{Map:Remove}0\mapsto n+1 \quad\text{and}\quad B\mapsto B^\z\ \ \text{for each nonzero block $B$}\end{equation}
 extends to an $\mathfrak S_n$-isomorphism $\Pi_n^{(d)}\to P_{n+1}^d$ when $P_{n+1}^d$ is considered as an 
$\mathfrak S_n$-poset by restricting its natural $\mathfrak S_{n+1}$-action to $\mathfrak S_{\{1,\ldots,n\}}$.  
As a consequence of this discussion and the main results of~\cite{Ehrenborg} we have the following.

\begin{theorem}[Calderbank-Hanlon-Robinson, Ehrenborg-Jung, Wachs]\label{Theorem:Skew}\ \\ 
$\widetilde{H}_{a(d)-1}(\Pi_n^{(d)}\diff\{\hat{1}\})$ is isomorphic to the $\mathfrak S_n$-Specht 
module of skew ribbon shape associated with $(n-d\cdot a(d), d,d,\ldots,d)$ when $d\nmid n$, and $0$ otherwise.
\end{theorem}

\begin{remark}
We have stated Theorem~\ref{Theorem:Skew} in a way that requires the observation that 
the skew shapes associated to $(n-d\cdot a(d), d,d,\ldots,d)$ and its reverse $(d,d,\ldots,d,n-d\cdot a(d))$ 
are related by $180^\circ$ rotation, and hence their Specht modules are $\mathfrak S_n$-isomorphic; see~\cite[Ex. 7.56]{Stanley:EC2}.
\end{remark}

\begin{figure}[hbt]
\center
\includegraphics[width=1\textwidth]{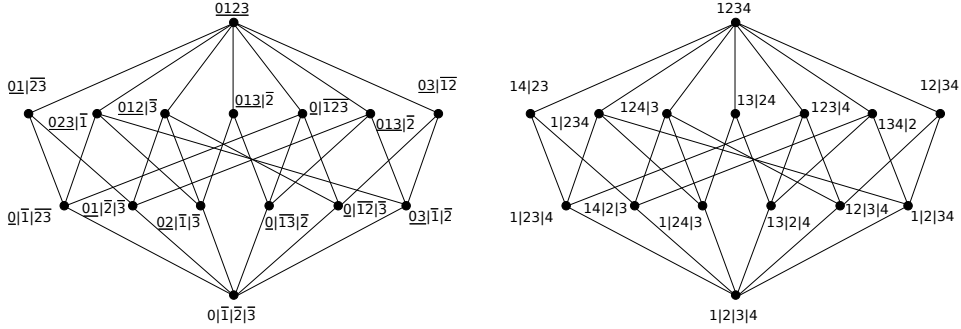}
\caption{The posets $\Pi_{n,1}$ and $\Pi_{n+1}$ for $n=3$}\label{Pi:Iso}
\end{figure}

Note that the $\mathfrak S_n$-ribbon representation of Theorem~\ref{Theorem:Skew} 
arises from a subarrangement of the reflection arrangement of $\mathfrak S_{n+1}$:
\[\Pi_n^{(d)}\subseteq \Pi_{n,1}\cong_{\mathfrak S_n}\Pi_{n+1}\cong_{\mathfrak S_n}
E(\mathfrak S_{n+1},1)\quad\text{with}\quad 
\mathfrak S_n=\mathfrak S_{\{1,\ldots, n\}},\]
where the isomorphism $\Pi_{n,1}\cong_{\mathfrak S_n}\Pi_{n+1}$ is obtained by extending~\eqref{Map:Remove}; see Figure~\ref{Pi:Iso}.    
We conjecture that these $\mathfrak S_n$-ribbon representations also appear in $E(\mathfrak S_n,\zeta_d)$, without the (a priori) mysterious shift of index.  
To make the assertion more precise, we introduce the new operation $B\mapsto B^\zero$ of \emph{zeroing}: given a block $B$, define $B^\zero$ to be 
the block obtained from $B$ by replacing each weight $w$ with $0$.  
An immediate observation is the following; see Figure~\ref{Z:Map}.

\begin{proposition}\label{d-divisible_map}
The map $Z:\Pi_n^d\to \Pi_n^{(d)}\subseteq \Pi_{n,1}$ given by zeroing each block
\[Z:B_0/B_1/\cdots/B_\ell\mapsto B_0/B_1^\zero/\cdots/B_\ell^\zero\]
is a rank-preserving surjective $\mathfrak S_n$-poset map.
\end{proposition}

\begin{figure}[hbt]
\center
\includegraphics[width=.8\textwidth]{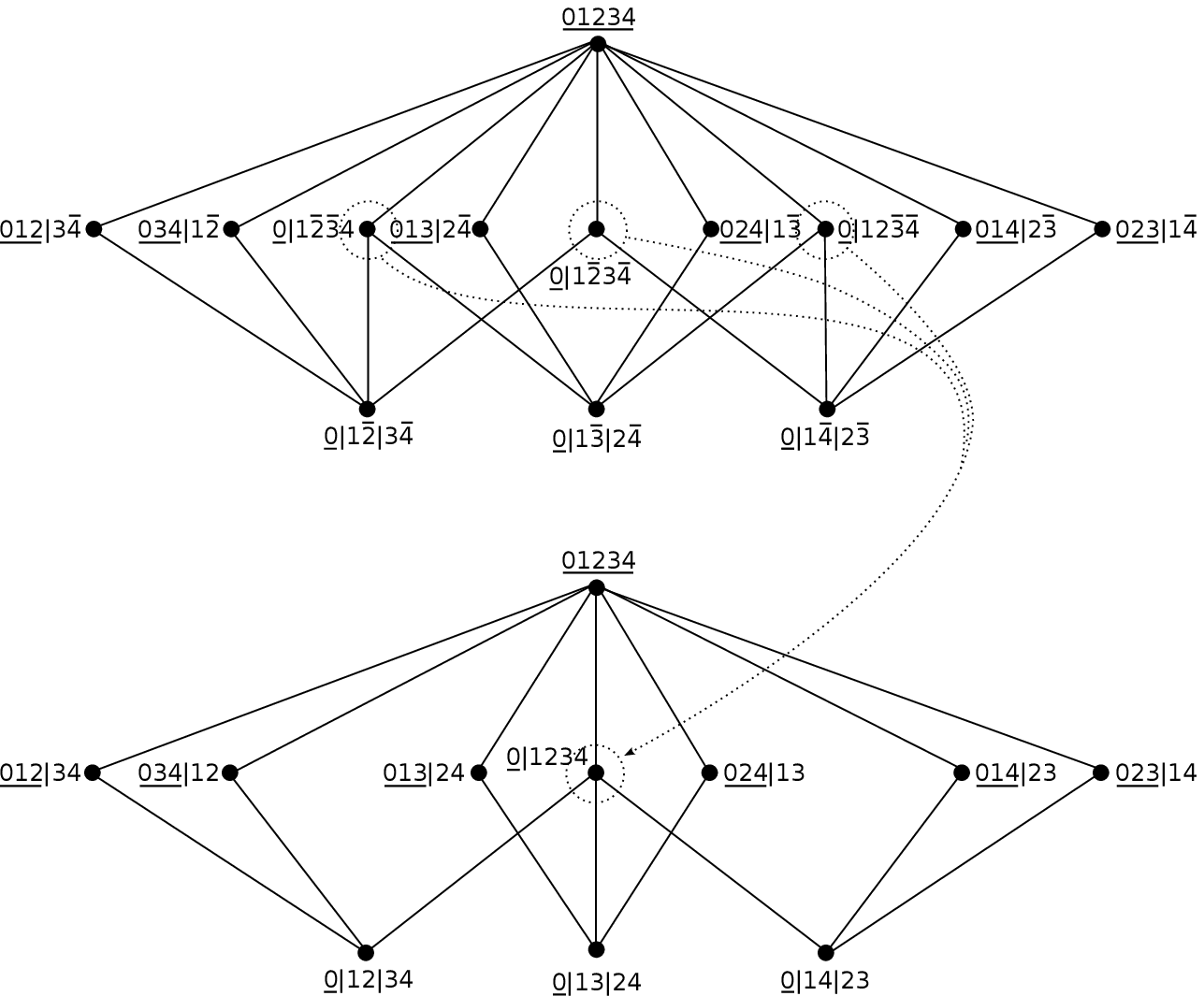}
\caption{The zeroing map $\Pi_4^2\to\Pi_4^{(2)}$}\label{Z:Map}
\end{figure}

\begin{conjecture}\label{Conjecture:Main}
Let $Z_*:\widetilde{H}_{a(d)-1}(\Pi_n^d\diff\{\hat{1}\})\rightarrow\widetilde{H}_{a(d)-1}(\Pi_n^{(d)}\diff\{\hat{1}\})$ 
be the homomorphism induced by the poset map of Proposition~\ref{d-divisible_map}.  Then $Z_*$ is surjective.
\end{conjecture}

Note that by Theorem~\ref{Theorem:Skew} the conjecture is trivially true when $d\mid n$.  
The conjecture also clearly holds when $r(\Pi_n^d)=1$, i.e., when $a(d)=1$.  In fact, 
the entire top homology module is straightforward in this case:

\begin{proposition}
Suppose $a(d)=1$.  Let $C_d$ be the 
cyclic group generated by the $d$-cycle $(1,2,\ldots, d)$ and let $\mathfrak S_{n-d}=\mathfrak S_{\{d+1,d+2,\ldots,n\}}$.  Then
\[\widetilde{H}_{a(d)-1}(\Pi_n^d\diff\{\hat{1}\})\cong_{\mathfrak S_n} \left(\Ind_{ C_d\times\mathfrak S_{n-d}}^{\mathfrak S_n}\triv\right)-\triv.\] 
\end{proposition}

\begin{proof}
$\mathfrak S_n$ acts transitively on the elements of $\Pi_n^d\diff\{\hat{1}\}$ in this case, and 
$C_d\times\mathfrak S_{n-d}$ is the stabilizer of the distinguished element of Proposition~\ref{Prop:Filter}.  
Hence the result.
\end{proof}

\section{Proof of Theorem~\ref{Thm:Main:KPi1}}\label{Section:KPi1}
The motivating case $d=1$ of Theorem~\ref{Thm:Main:KPi1} has a long history, 
starting with  Fadell and Neuwirth~\cite{Fadell:Neuwirth} showing that $\mathcal M(\mathfrak S_n,1)$ is 
$K(\pi,1)$.  Deligne~\cite{Deligne} established Brieskorn's conjecture that $\mathcal M(W,1)$ is $K(\pi,1)$ 
for all (complexified) real reflection 
groups.  Orlik and Solomon~\cite{Orlik:Solomon} subsequently showed that $\mathcal M(W,1)$ is $K(\pi,1)$ for 
every \emph{Shephard group} $W$ (see~\cite[\S6.6]{Orlik:Terao}) by showing that 
$\mathcal M(W,1)/W$ is the same as $\mathcal M(W',1)/W'$ for some associated real reflection group $W'$ 
and then invoking Deligne's result, while 
Nakamura~\cite{Nakamura} established the result for the family $G(m,p,n)$, leaving only six exceptional cases.  
Bessis~\cite{Bessis} 
resolved the question in the remaining cases, while also recovering many of the previous results with his approach.
In particular, he obtained a new proof of Deligne's result for real reflection groups.

\begin{theorem}[Bessis]\label{BessisTheorem}
For any reflection group $W$, the reflection complement $\mathcal M(W,1)$ is $K(\pi,1)$, i.e., its universal 
cover is contractible.
\end{theorem}
\noindent
The most one could hope for in general is that $\mathcal M(W,d)$ is $K(\pi,1)$ whenever $a(d)$ is 
$n$ or $n-1$, since any complex arrangement $\mathcal A$ whose complement is $K(\pi,1)$ must 
contain a hyperplane; see~\cite[Cor. 3.2]{Bjorner:Welker} for a much stronger result.  
Since Proposition~\ref{Prop:Springer:Transitive}\eqref{Springer:1} tells us that 
$\mathcal A(W,d)$ is in fact a hyperplane arrangement when $a(d)$ is 
$n$ or $n-1$, one may also hope it is \emph{free} in the sense of Terao~\cite{Terao}, who 
established freeness for reflection arrangements $\mathcal A(W,1)$.

\begin{theorem}[Terao]\label{TeraoTheorem}
For any reflection group $W$, the reflection arrangement $\mathcal A(W,1)$ is free.
\end{theorem}

Recall that a collection of hyperplanes $\mathcal A$ in a vector space $V$ is a \emph{free arrangement} 
if the $S$-module of derivations $\delta:S\to S$ that satisfy $\delta(\alpha_H)\in \alpha_H S$ for 
every hyperplane $H=\ker(\alpha_H)$ in $\mathcal A$ is free, where $S=S(V^*)$ as in \S\ref{Section:General}.  
In particular, 
the \emph{null arrangement} $\mathcal N_V$ of $V$, which contains no subspaces at all, is a free arrangement.  
Less obvious is that any finite collection of lines through the origin in $\mathbb C^2$ is free, 
which follows from \emph{Saito's criterion}.  (It is also straightforward~\cite[Examples 4.20 and 5.4]{Orlik:Terao} 
to see that they are $K(\pi,1)$, 
by which we mean that their complements have the property.)   
Lastly, 
recall that the product
\[\mathcal A_1\times\mathcal A_2:=\{H_1\oplus V_2\ :\ H_1\in\mathcal A_1\}\cup\{V_1\oplus H_2\ :\ H_2\in \mathcal A_2\}\]
of two arrangements is free if and only if both 
$\mathcal A_1$ and $\mathcal A_2$ are free; see~\cite{Orlik:Terao}.

\pagebreak
The crux of Theorem~\ref{Thm:Main:KPi1} is the following observation.
\begin{lemma}\label{KPi1:Lemma}
Let $W$ be a complex reflection group and let $d>0$.
\begin{enumerate}[(i)]
\item  If $a(d)=n$, then $\mathcal A(W,d)=\mathcal A(W,1)$ by Corollary~\ref{Cor:Springer}.\label{K:Pi:1:n}
\item  If $a(d)=n-1$ and $W$ is irreducible, then one of the following 
holds:\label{K:Pi:1:n1}
\begin{enumerate}[a.]
\item  $W$ has rank $2$, so that $\mathcal A(W,d)$ is a set of 
lines through the origin in $\mathbb C^2$.\label{K:Pi:1:n1:2}
\item  $W=G_{25}$, $d=2$ (or $6$), and $\mathcal A(W,d)=\mathcal A(W',1)$ for $W'=G(3,3,3)$.\label{K:Pi:1:n1:G25}
\item  $W=W(A_3)$, $d=2$, and $\mathcal A(W,d)$ is defined by $z_1z_2z_3$.\label{K:Pi:1:n1:A3}
\item  $W=G(m,p,n)$, $d\mid m$ and $d\nmid \frac{nm}{p}$, and $\mathcal A(W,d)$ is defined 
by $z_1z_2\cdots z_n$. \label{K:Pi:1:n1:Infinite}
\end{enumerate}
\end{enumerate}
\end{lemma}

\begin{proof}
\eqref{K:Pi:1:n} is clear.  Consulting the classification~\cite{Shephard:Todd}, 
one finds that the pairs $(W,d)$ listed in~\eqref{K:Pi:1:n1} 
are the only cases such that $a(d)=n-1$.  Case~\eqref{K:Pi:1:n1:2} is clear.  
For~\eqref{K:Pi:1:n1:A3} and~\eqref{K:Pi:1:n1:Infinite}, 
set $f_i:=e_i(z_1^m,z_2^m,\ldots, z_n^m)$ for $1\leq i\leq n-1$ and set $f_n:=e_n(z_1^{m/p},z_2^{m/p},\ldots, z_n^{m/p})$, where 
here $e_i$ denotes the $i$th \emph{elementary symmetric polynomial} in $n$ variables.  These form a set of basic invariants 
for $G(m,p,n)$, the degrees of which are $m,2m,\ldots, (n-1)m, nm/p$, and only the last invariant polynomial 
$f_n=(z_1z_2\cdots z_n)^{m/p}$ has degree $d_n=nm/p$ not divisible by $d$.  Applying 
Proposition~\ref{Prop:Springer:Intersection} thus gives~\eqref{K:Pi:1:n1:A3} and~\eqref{K:Pi:1:n1:Infinite}.  
Similarly, for $G_{25}$ we may choose~\cite[eq. 9]{Maschke} coordinates and basic invariants $f_1,f_2,f_3$ (of degrees 6,9,12) 
such that $f_2$ is the polynomial $(z_1^3-z_2^3)(z_1^3-z_3^3)(z_2^3-z_3^3)$
that (by~\eqref{Eq:Hyp:1}) defines $\mathcal A(W',1)$, then apply Proposition~\ref{Prop:Springer:Intersection}.
\end{proof}

\noindent{\bf Proof of Theorem~\ref{Thm:Main:KPi1}.}  
As noted above, if $\mathcal A(W,d)$ is $K(\pi,1)$ (resp. free), then either $a(d)=n$ or $a(d)=n-1$.  
When $a(d)=n$, the converse follows from Theorem~\ref{BessisTheorem} (resp. Theorem~\ref{TeraoTheorem}), 
so it remains to see that the converse also holds when $a(d)=n-1$.  In this case,
decompose $W$ and $V:=\mathbb C^n$ as in Proposition~\ref{Prop:Decomposition} so that 
$W\cong W_1\times\cdots\times W_k$.  Then 
$\mathcal A(W,d)=\mathcal A(W_i,d)\times \mathcal N_{V_i^\perp}$ for some index $i$, and 
thus $\mathcal A(W,d)$ is $K(\pi,1)$ (resp. free) if and only if $\mathcal A(W_i,d)$ is $K(\pi,1)$ (resp. free).  
Now employ Lemma~\ref{KPi1:Lemma}\eqref{K:Pi:1:n1}.
\qed

\begin{remark}
V. Reiner pointed out that one may alternatively consider the question of real $K(\pi,1)$ complements.  
That is, for which (uncomplexified) real reflection groups 
$W\subset \GL(\mathbb R^n)$ is the \emph{real} complement 
\[\mathcal M_{\mathbb R}(W,2)\Def\mathbb R^n\diff \textstyle{\bigcup_{g\in W} V(g,-1)} \]
a $K(\pi,1)$ space?  Because such a complement is $K(\pi,1)$ only if the arrangement has codimension 2, 
one reduces the question to Weyl groups
$W(A_4)$, $W(A_5)$, and $W(E_6)$ via the classification.  
Though we have not explored these cases, 
we note that a similar question was answered in~\cite{Davis}, where Davis, Januszkiewicz, and Scott established a conjecture 
of Khovanov~\cite{Khovanov} asserting that the real complement of 
any $W$-invariant codimension-2 subarrangement of $\L_W$ is $K(\pi,1)$ when $W$ is a real reflection group.  
Khovanov, in turn, was motivated by (and answered positively) Bj\"orner's question~\cite[\S 13.7]{Bjorner:Europe} 
of whether the complement of the \emph{$3$-equal} arrangement 
\[\mathcal A_{n,3}:=\{{\mathbf x}\in\mathbb R^n\ :\ x_i=x_j=x_k\ \ \ \text{for some triple}\ \ \ i<j<k\}\] 
is a $K(\pi,1)$ space.
\end{remark}

\begin{landscape}
\begin{table}
\begin{center}
\begin{tabular}{l@{\text{ }}l@{\hskip .5cm}c@{\hskip .5cm}l@{\hskip .5cm}l@{\hskip 1.0cm}lllll@{\hskip 1.0cm}cllll}
\toprule
&&&&&\multicolumn{5}{l}{\!\!\!Rank sizes of $E(W,\zeta_d)$}&\multicolumn{5}{l}{\!\!\!Betti numbers of $\Delta(\overline{E(W,\zeta_d)})$}\\
$W$&&$a(d)$&$d$&$W(d)$&$r_0$&$r_1$&$r_2$&$r_3$&$r_4$&\footnotesize{$|C_{-1}|$}&\footnotesize{$|C_0|$}&\footnotesize{$|C_1|$}&\footnotesize{$|C_2|$}&\footnotesize{$|C_3|$}\\
\midrule
 $G_{25}$&    &  2                              &            2,6                          &             $G_5$                         &  9& 48     &1  &&&           1& 57& 72& &   \\ 
 $G_{28}$&&      2                              &            3,6                          &             $G_5$                         &  16& 64    & 1 &&&           1& 80& 128& &    \\
        &    & 2                              &            4                            &             $G_8$                         &  12& 36    & 1 &&&           1 &48& 72&  &     \\       

 $G_{30}$&&      2                              &            3,6                          &             $G_{20}$                      &  40& 400    & 1 &&&           1& 440& 800& &     \\
        &    & 2                              &            4                            &             $G_{22}$                      &  60& 900    & 1 &&&           1& 960& 1800& &       \\     
        &    & 2                              &            5,10                         &             $G_{16}$                      &  24& 144    & 1 &&&           1& 168& 288&  &        \\

 $G_{31}$&    & 2                              &            3,6,12                       &             $G_{10}$                      &  160& 880   & 1  &&&          1& 1040& 2240& &     \\
        &    & 2                              &            8                            &             $G_{9}$                       &  240&1800   & 1 &&&           1& 2040& 4320& &         \\

 $G_{32}$&    & 2                              &            4,12                         &             $G_{10}$                      &  540& 1980  & 1  &&&          1& 2520& 7560& &      \\

 $G_{33}$&    & 2\makebox[0cm]{\, *}           &            4                            &             $G_{6}$                       &  540& 1860  & 1  &&&          1& 2400 & 5400 & &      \\
        &    & 3                              &            3,6                          &             $G_{26}$                      &  40& 600& 925 & 1 &&          1& 1565& 7920& 6720&       \\     

 $G_{34}$&    & 2\makebox[0cm]{\, *}           &            4,12                         &             $G_{10}$                      &  34020& 47250 & 1 &&&         1 & 13230 & 68040 &&      \\
 \makebox[0cm]{{\hskip -2.2cm}(type $E_6$)}$G_{35}$& &      2                              &            4                            &             $G_{8}$          &  540& 540     & 1 &&&         1& 1080& 3240&  &   \\
        &    & 2                              &            6                            &             $G_{5}$                       &  720& 1680    & 1 &&&         1& 2400& 5760& &      \\     
        &    & 3                              &            3                            &             $G_{25}$                      &  80& 480& 330 & 1 &&          1& 890& 5040& 4800&         \\
        &    & 4                              &            2                            &             $G_{28}$                      &  45 & 810& 2730 & 1611 & 1&   1& 5196& 58590& 130680& 77760         \\
 \makebox[0cm]{{\hskip -2.2cm}(type $E_7$)}$G_{36}$& &      2\makebox[0cm]{\, *}           &            4                            &             $G_{8}$          & 11340& 1890       &1&&&       1& 13230& 68040 & &   \\
        &    & 3                              &            3,6                          &             $G_{26}$                      & 2240& 18480& 13272&1 &&       1& 33992& 322560& 376320&  \\     
 \makebox[0cm]{{\hskip -2.2cm}(type $E_8$)}$G_{37}$& &      2                              &            5,10                         &             $G_{16}$         & 1161216    &?& 1 &&&          1&?&$11\times r_0$&&        \\
        &    & 2                              &            8                            &             $G_{9}$                       & 3628800    &?& 1 &&&         1&?&$18\times r_0$&&    \\  
        &    & 2                              &            12                           &             $G_{10}$                      & 2419200    &?& 1 &&&         1&?&$14\times r_0$&&    \\     
        &    & 4                              &            3,6                          &             $G_{32}$                      & 4480    & 89600 & 319200 & 103040 &1&            1&$\sum_0^3r_i$&?&?&$6720\times r_0$ \\
        &    & 4                              &            4                            &             $G_{31}$                      & 15120   & 453600 & 2205000  & 1519560 & 1 &      1&$\sum_0^3r_i$&?&?&$22680\times r_0$\\
\bottomrule
 \end{tabular}
\tableskip
\caption{
An expanded version of Table~\ref{Table:Exceptionals}.  
}
\label{Big:Table}
\end{center}
\end{table}
\end{landscape}

{}
\end{document}